\documentclass{amsart}

\usepackage{amssymb}
\usepackage[all]{xy}
\usepackage{caption}
\usepackage[pdftex]{hyperref}
\usepackage{graphicx}
\usepackage{array}
\usepackage{tikz}
\usepackage{amsmath}

\newtheorem{thm}{Theorem}
\newtheorem*{thm*}{Theorem}

\newtheorem{conj}[thm]{Conjecture}
\newtheorem{prop}[thm]{Proposition}
\newtheorem{cor}[thm]{Corollary}
\newtheorem{rem}[thm]{Remark}
\theoremstyle{definition}

\newcommand\scalemath[2]{\scalebox{#1}{\mbox{\ensuremath{\displaystyle #2}}}}

\newcommand{\pgl}{\PGL(2,\mathbb{C})}

\newcommand{\sldos}{\SL(2,\mathbb{C})}
\DeclareMathOperator{\tr}{tr}
\DeclareMathOperator{\Id}{Id}

\DeclareMathOperator{\GL}{GL}

\DeclareMathOperator{\SL}{SL}

\DeclareMathOperator{\PGL}{PGL}

\DeclareMathOperator{\Gr}{Gr}

\newcommand{\x}{\times}
\newcommand{\CC}{\mathbb{C}} 
\newcommand{\RR}{\mathbb{R}} 
\newcommand{\ZZ}{\mathbb{Z}} 
\newcommand{\too}{\longrightarrow}

\newcommand{\im}{\mathrm{im}\,}

\newcommand{\oX}{\overline{X}{}}
\newcommand{\oW}{\overline{W}{}}
\newcommand{\oZ}{\overline{Z}{}}

\author[J. Mart\'{\i}nez]{Javier Mart\'{\i}nez}
\address{Escuela de Arquitectura, Ingenier\'ia y Dise\~no, Universidad Europea,
Calle Tajo s/n, 28670 Villaviciosa de Od\'on, Madrid, Spain}
\email{javier.martinez2@universidadeuropea.es}

\title[E-polynomials of $\PGL(2,\CC)$-character varieties]{E-polynomials 
of  $PGL(2,\CC)$-character varieties of surface groups}

\date{12th May 2017}

\subjclass[2010]{Primary: 14C30. Secondary: 14D20, 14L24, 32J25}

\keywords{Moduli spaces, E-polynomial, character variety, surface group}

\begin{document}

\begin{abstract}
We compute the E-polynomials of the $\pgl$-character varieties associated to surfaces of genus $g$ with one puncture, for any holonomy around it, and compare it with its Langlands dual case, $\sldos$. The study is based on the stratification of the space of
representations and on the analysis of the behaviour of the E-polynomial under fibrations.
\end{abstract}

\maketitle

\section{Introduction}\label{sec:introduction}
The non-abelian Hodge correspondence for a smooth complex algebraic curve $X$ of genus $g$ stablishes isomorphisms between three moduli spaces: the moduli space of $G$-Higgs bundles on $X$ $\mathcal{M}_{Dol}(G)$, the moduli space of flat $G$-connections $\mathcal{M}_{dR}(G)$ and the character variety or Betti moduli space $\mathcal{M}_{B}(G)$, where in this paper $G=SL(n,\CC),GL(n,\CC),PGL(n,\CC)$. They are homeomorphic and the cohomology of these spaces has mainly been studied from the perspective of $\mathcal{M}_{Dol}$; the Poincar\'e polynomials were computed in \cite{hitchin:1987} for $G=\sldos$, in \cite{gothen:1994} for $G=SL(3,\CC)$ and recently a solution for all $n$ has been obtained \cite{Schiffmann:2014} \cite{MozSch}, as well as recursive formulas for the motive of the moduli of Higgs bundles \cite{garciaprada-heinloth:2013}.

A key aspect is that the Riemann-Hilbert map taking a flat connection to its monodromy is not algebraic, which is reflected in the fact that the mixed Hodge structure on the cohomology of the character variety is different from the Dolbeault and de Rham cases, where they are agree and are pure \cite{hausel-thaddeus:2003}. A explanation of this phenomenon in terms of the perverse filtration associated to the Hitchin map was conjectured and proved for $n=2$ in \cite{deCat-Hau-Migl}. Motivated by mirror symmetry considerations \cite{hausel:2005}, Hausel and Rodriguez-Villegas \cite{hausel-rvillegas:2008} computed the E-polynomials of $\mathcal{M}_{B}$ for $G=GL(n,\CC)$, a specialization of the mixed Hodge polynomial that is built from the mixed Hodge numbers of the variety. Their method was arithmetic and produced generating functions for the E-polynomials based on the count of points of $\mathcal{M}_{B}$ over finite fields, due to a theorem of Katz. Following similar ideas, Mereb dealt with the $SL(n,\CC)$ case in \cite{mereb:2010}, and the case of a punctured Riemann surfaces with semisimple conjugacy classes at the punctures for $G=GL(n,\CC)$ was treated in \cite{hausel-letellier-rvillegas:2011}. Recently, Baraglia and Hekmati have provided explicit formulas for $G=\sldos,GL(2,\CC)$ and $SL(3,\CC),GL(3,\CC)$ for several finitely generated groups using also arithmetic methods \cite{baraglia-hekmati:2016}, and the case of wild character varieties has also been studied by Hausel, Mereb and Wong \cite{hausel-mereb-wong:2016}.

A different approach to the problem was started in \cite{lomune} for $G=\sldos$ and the low genus cases $g=1,2$, motivated by explicit descriptions \cite{mun1} of other character varieties. The technique was geometric and it was based on the study of the behaviour of the E-polynomials for fibrations which are locally trivial in the analytic topology, but not in the Zariski topology. Explicit stratifications of the representation spaces using trace maps were also used. Their main results were explicit formulas for the E-polynomials of character varieties of punctured Riemann surfaces of genus $1,2$ where the monodromy around the puncture is arbitrary, not necessarily semisimple.  The techniques were used in \cite{lomu} for two punctures and in \cite{CavLaw:2014}\cite{law-mun:2016} for free groups, and they were later developed in \cite{mamu} to deal with fibrations over bases of complex dimension two, where the case $g=3$ was solved. Finally, in \cite{mamu2} the arbitrary genus polynomials were obtained inductively and formulas generalizing the results of \cite{lomune},\cite{mamu} were given.
The behaviour of the E-polynomial of the parabolic character variety,
$$
\mathcal{M}_{\lambda}(\sldos)=\lbrace (A_{1},B_{1},\ldots, A_{g},B_{g})\in \sldos^{2g} \mid \prod_{i=1}^{g}[A_{i},B_{i}]=\begin{pmatrix} \lambda & 0 \\ 0 & \lambda^{-1}\end{pmatrix} \rbrace//\sldos
$$
when $\lambda$ varies in $\CC-\lbrace 0, \pm 1 \rbrace$ was crucial, and it was encoded in a polynomial $R(\mathcal{M}_{\lambda})$ with coefficients in the representation ring $R(\ZZ_2\times \ZZ_2)$, the \emph{Hodge monodromy representation}.

In this paper, the same approach is used to determine the E-polynomials of $\pgl$-character varieties of punctured Riemann surfaces with arbitrary monodromy at the puncture. That is, for any $C\in \sldos$, the character varieties that we will consider are the spaces
$$
\mathcal{M}_{C}(\pgl)=\lbrace (A_{1},B_{1},\ldots, A_{g},B_{g})\in \pgl^{2g} \mid \prod_{i=1}^{g}[A_{i},B_{i}]=C \rbrace // Stab(C)
$$
which are $\ZZ_2^{2g}$-quotients of the corresponding $\sldos$-character varieties, given by the action by multiplication of the square roots of unity on each matrix element, i.e. $\mathcal{M}_{C}(\pgl)=\mathcal{M}_{C}(\sldos)/\ZZ_2^{2g}$. The main results are the following:
\begin{thm} \label{thm:epolysmodulis} For all $g\geq 1$,
\begin{align*}
e(\mathcal{M}_{\Id}^{g}) & =  (q^{3}-q)^{2g-2}+(q^{2}-1)^{2g-2}-q(q^{2}-q)^{2g-2} -q^{2g-2},  \\
& +\frac{1}{2}q^{2g-1}((q-1)^{2g-2}+(q+1)^{2g-2})+\frac{1}{2}q((q+1)^{2g-1}+(q-1)^{2g-1}), \\
e(\mathcal{M}_{-\Id}^{g}) & = (q^{3}-q)^{2g-2}+(q^{2}-1)^{2g-2}-\frac{1}{2}((q^{2}+q)^{2g-2}+(q^{2}-q)^{2g-2}), \\
e(\mathcal{M}_{J_{+}}^{g}) & = (q^{2}-1)(q^{3}-q)^{2g-2}+\frac{1}{2}(-q(q+1)(q^{2}+q)^{2g-2}+(q-2)(q-1)(q^{2}-q)^{2g-2}), \\
e(\mathcal{M}_{J_{-}}^{g}) & =  (q^{2}-1)(q^{3}-q)^{2g-2}+\frac{1}{2}((q+1)(q^{2}+q)^{2g-2}-(q-1)(q^{2}-q)^{2g-2}), \\
e(\mathcal{M}_{\lambda}^{g}) & = (q^{2}+q)(q^{3}-q)^{2g-2}+(q+1)(q^{2}-1)^{2g-2} -q(q^{2}-q)^{2g-2}.
\end{align*}
for $J_+=\begin{pmatrix} 1& 1\\ 0 &1\end{pmatrix}$,
$J_-=\begin{pmatrix} - 1& 1\\ 0 &- 1\end{pmatrix}$ and
$\xi_\lambda=\begin{pmatrix} \lambda & 0\\ 0 &\lambda^{-1}\end{pmatrix}$, $\lambda\neq 0,\pm 1$,
and with $q=\, uv$. 
\end{thm}
A variety is said of \emph{balanced type} (also called Hodge-Tate type in the literature) when $H_{c}^{k,p,q}=0$ for $p\neq q$. As a consequence of the tools that are used, it is also obtained
\begin{thm}
The character varieties $\mathcal{M}_{C}(\pgl)$ are of balanced type, for any conjugacy class $C$.
\end{thm}

The E-polynomial information of the character varieties for punctured curves of genus $g$ and arbitrary monodromy at the punctures is encoded in 6 polynomials, $v^{g}=(\tilde{e}_{0}^{g},\tilde{e}_{1}^{g},\tilde{e}_{2}^{g},\tilde{e}_{3}^{g}, \tilde{a}^{g},\tilde{b}^{g})$, that correspond to the representation spaces for $C=\Id,-\Id,J_{+},J_{-}$ and the monodromy information of the parabolic family $\mathcal{M}_{\lambda}^{g}$. The genus $1$ case provides what we call basic blocks, which are computed explicitly in Section \ref{sec:basicblocks}. In Section \ref{chapter:highgenus}, they are used as basis for an induction procedure that provides the formulas of Theorem \ref{thm:epolysmodulis}. The induction step is given by a $6 \times 6$ matrix $M$ with coefficients in $\ZZ[q]$ that represents the handle attachment $X^{g}=X^{g-1} \# X^{1}$ at the level of the E-polynomials of the representation spaces, since $v^{g}=Mv^{g-1}$.

Finally, if we look at the E-polynomials of the representation spaces for $\sldos$ and $\pgl$, we uncover the following pattern. The parabolic representation varieties
$$
\oX_{4,\lambda}=\lbrace (A_{1},B_{1},\ldots, A_{g},B_{g})\in \sldos^{2g} \mid \prod_{i=1}^{g}[A_{i},B_{i}]=\begin{pmatrix} \lambda & 0 \\ 0 & \lambda^{-1}\end{pmatrix} \rbrace
$$
can be seen as the fibres of a total space $\oX_{4}$ that fibres over $\CC- \lbrace 0 \pm 1 \rbrace$. There is a $\ZZ_2$-action on $\oX_{4}$ compatible with the fibration that takes $\lambda$ to $\lambda^{-1}$ on the base and conjugates the representation by interchanging the eigenvectors, giving rise to a fibration $\oX_{4}/\ZZ_2 \rightarrow \CC- \lbrace \pm 2 \rbrace$. As we have already mentioned, the behaviour of the E-polynomial of the monodromy of this fibration is controlled by a polynomial with coefficients in the representation ring $R(\ZZ_2 \times \ZZ_2)$,
\begin{equation} \label{eqn:Hodgemonintrox4z2}
R(\oX_{4}/\ZZ_2)=a^{g}T+b^{g}S_{2}+c^{g}S_{-2}+d^{g}S_{0}\in R(\ZZ_2\times \ZZ_2)[q]
\end{equation}
where $a^{g},b^{g},c^{g},d^{g}\in \ZZ[q]$ and $T,S_{0},S_{-2},S_{2}$ are elements in $R(\ZZ_2\times \ZZ_2)$, as it was shown and explained in \cite{mamu2}. Here, its $\pgl$-version is computed in Theorem \ref{thm:Polynomialseig}. With the information from Theorem \ref{thm:epolysmodulis}, analyzing the E-polynomials for the Langlands dual groups $G=\sldos$ and $\hat{G}=\pgl$, we observe
\begin{thm}
The difference between the E-polynomials of the $G,\hat{G}$-character varieties associated to punctured Riemann surfaces can be expressed in terms of the coefficients of the Hodge monodromy representation of the fibration $\oX_{4}/\ZZ_2 \rightarrow \CC-\lbrace \pm 2 \rbrace$,
$$
R(\oX_{4}/\ZZ_2)=a^{g}T+b^{g}S_{2}+c^{g}S_{-2}+d^{g}S_{0}
$$
These differences are the following 
\begin{align}
e_{0}-\tilde{e}_{0} & = (q^3-q)(2^{2g}-1)\left(\frac{(q^{2}-q)^{2g-2}+(q^{2}+q)^{2g-2}}{2}\right)=qc^{g}+b^{g}, \nonumber\\
e_{1}-\tilde{e}_{1} & = (q^3-q)(2^{2g}-1)\left(\frac{(q^{2}-q)^{2g-2}-(q^{2}+q)^{2g-2}}{2}\right)=qb^{g}+c^{g}, \\
e_{2}-\tilde{e}_{2} & = (2^{2g}-1)\left(\frac{(q^{2}-q)^{2g-1}-(q^{2}+q)^{2g-1}}{2}\right)=b^{g}, \nonumber\\
e_{3}-\tilde{e}_{3} & = (2^{2g}-1)\left(\frac{(q^{2}-q)^{2g-1}+(q^{2}+q)^{2g-1}}{2}\right)=c^{g}, \nonumber \\
e_{4,\lambda}-\tilde{e}_{4,\lambda} & = (2^{2g}-1)(q^{2}-q)^{2g-1} =b^{g}+c^{g}. \nonumber
\end{align}
where $e_{i},\tilde{e}_{i}, i=1\ldots 4$ are the E-polynomials of the $\sldos$, $\pgl$-representation spaces for $C=\Id, -\Id, J_{+},J_{-}, \xi_{\lambda}$ respectively.
\end{thm}
We end up with Conjecture \ref{conj:Hodgemonmirrorsymmetry}, which asks whether this phenomenon is still true for $SL(n,\CC)$ and $PGL(n,\CC)$, $n\geq 3$. In this case, the fibration $\oX_{4}$ is replaced by a fibration $\oX_{\boldsymbol{\lambda}}$ of representation spaces of punctured Riemann surfaces whose monodromy at the puncture is fixed and semisimple, given by a matrix $C$ determined by eigenvalues $\boldsymbol{\lambda}=(\lambda_{1},\ldots,\lambda_{n})$. It is natural to expect that the mirror symmetric differences are given in terms of the Hodge monodromy representation of the quotient fibration $R(\oX_{\boldsymbol{\lambda}}/S_{n})$ and that the expressions obey a symmetric pattern in similar way to what occurs for the $n=2$ case.

\noindent\textbf{Acknowledgements.} 
I would like to thank my phD thesis advisor Vicente Mu\~noz for his advice and for many enlightening discussions related with this work. I would like to also thank Marina Logares and Peter Newstead for useful conversations related to this paper.
\section{E-polynomials}\label{sec:E-polynomials}

A pure Hodge structure of weight $k$ consists of a finite dimensional complex vector space
$H$ with a real structure and a decomposition $H=\bigoplus_{k=p+q} H^{p,q}$
such that $H^{q,p}=\overline{H^{p,q}}$, where the bar means complex conjugation on $H$.
An equivalent definition can be obtained by replacing the direct sum decomposition by the Hodge filtration, which is a descending filtration by complex subspaces $F^{p}$ subject to the condition  $F^{p}\cap \overline{F^{k-(p-1)}}=0$ for all $p$. It this case, we can define $Gr^{p}_{F}(H):=F^{p}/F^{p+1}=H^{p,k-p}$. A classical result is the Hodge decomposition, that asserts that the cohomology groups $H^{k}(X,\CC)$ of a complex K\"ahler manifold have a pure Hodge structure of weight $k$.

A more general concept is given by the notion of a mixed Hodge structure. It consists of a finite dimensional complex vector space $H$ with a real structure and two filtrations:
an ascending filtration $W_{\bullet}$ $0 \subset \ldots \subset W_{k-1}\subset W_k \subset \ldots \subset H$
(defined over $\RR$) called the weight filtration, and another descending filtration $F^{\bullet}$, called the Hodge filtration, such that $F$ induces a pure Hodge structure
of weight $k$ on each graded piece $\Gr^{W}_{k}(H)=W_{k}/W_{k-1}$. We define
 $$
 H^{p,q}:= \Gr^{p}_{F}\Gr^{W}_{p+q}(H),
 $$
and define the Hodge numbers $h^{p,q} :=\dim_{\CC} H^{p,q}$.

Let $Z$ be any quasi-projective algebraic variety, possibly non-smooth or non-compact. 
The cohomology groups $H^k(Z)$ and the cohomology groups with compact support  
$H^k_c(Z)$ are endowed with mixed Hodge structures \cite{Deligne2,Deligne3} for each $k$. 
We can define the {\em Hodge numbers} of $Z$ by
$h^{k,p,q}_{c}(Z) = h^{p,q}(H_{c}^k(Z))$.
The Hodge-Deligne polynomial, or \emph{E-polynomial}, is defined as the alternate sum
 $$
 e(Z)=e(Z)(u,v):=\sum _{p,q,k} (-1)^{k}h^{k,p,q}_{c}(Z) u^{p}v^{q} \in \ZZ[u,v],
 $$
a specialization at $t=1$ of the mixed Hodge polynomial $h(Z)=\sum h^{k,p,q}_{c}(Z)u^{p}v^{q}t^{k}$ that encodes all mixed Hodge numbers. A relevant case that will appear in our setup is when $h_c^{k,p,q}=0$ for $p\neq q$. In that case, the polynomial $e(Z)$ depends only on the product $q:=uv$ and the variety is said to be {\it of balanced type}, or of {\it Hodge-Tate type}. An example is $e(\CC^n)=q^n$.

\subsubsection*{$\ZZ_2$-actions} When there is an action of $\ZZ_2$ on $Z$, we have
polynomials $e(Z)^+, e(Z)^-$, the E-polynomials of the invariant and the anti-invariant parts of the cohomology of $Z$ respectively. That is
$$
e(Z)^+=e(Z/\ZZ_2), \quad e(Z)^-=e(Z)-e(Z)^+.
$$
We will repeatedly use the following: if we have a fibration
$F\to Z\to B$, locally trivial in the analytic topology, with trivial monodromy and an action of $\ZZ_2$ 
on the fibration, we have an equality (see \cite[Proposition 2.6]{lomune}) 
 \begin{equation}\label{eqn:+-}
 e(Z/\ZZ_2) =e(F)^+e(B)^+ +e(F)^-e(B)^-\, .
  \end{equation}
  
\subsubsection*{Fibrations}
Assume that
 \begin{equation}\label{fibration}
  F \longrightarrow Z \overset{\pi}{\longrightarrow} B 
 \end{equation}
is a fibration locally trivial in the analytic topology, and let $F$ be of balanced type.
The fibration defines a local system $\mathcal{H}^{k}_{c}$, of fibre $H^{k}_{c}(F_{b})$, where $b\in B$, $F_{b}=\pi^{-1}(b)$. There is a monodromy representation $\rho$ associated to the local system, given by
 \begin{equation}\label{eqn:exxtra}
 \rho : \pi_{1}(B) \longrightarrow  \GL(H^{k,p,p}_c(F)).
 \end{equation}
Suppose that the  monodromy group $\Gamma=\im(\rho)$ is a finite group. 
Then the vector spaces $H_c^{k,p,p}(F)$ are modules over the representation ring $R(\Gamma)$, and we can consider the following well-defined element, that we call the  {\em Hodge monodromy representation},
  \begin{equation}\label{eqn:Hodge-mon-rep}
  R(Z) := \sum (-1)^k H_c^{k,p,p}(F)\,  q^p \in R(\Gamma)[q] \, .
  \end{equation}
As the monodromy representation (\ref{eqn:exxtra}) has finite image, there is a finite covering
$B_\rho \to B$ such that the pull-back fibration has trivial monodromy. We have the
following result.

\begin{thm}[{\cite[Theorem 2]{mamu}}] \label{thm:general-fibr}
 Suppose that $B_\rho$ is of balanced type. Then $Z$ is of balanced type. 
There is a $\ZZ[q]$-linear map
 $$
 e: R(\Gamma)[q] \to \ZZ[q]
 $$
satisfying the property that $e(R(Z))=e(Z)$.
\end{thm}

Furthermore, suppose that 
the  monodromy group $\Gamma=\im(\rho)$ is a finite abelian group. 
Let $S_1,\ldots, S_N$ be the
irreducible representations of $\Gamma$, which are generators of $R(\Gamma)$ as a free abelian group. 
Write $s_i(q)=e(S_i)$, $1\leq i\leq N$. Rewriting the Hodge monodromy representation (\ref{eqn:Hodge-mon-rep}) as
  $$
 R(Z)= a_1(q) S_1 + \ldots a_N(q) S_N,
 $$
then by (\ref{thm:general-fibr}) 
 $$
 e(Z)= a_1(q) s_1(q) + \ldots +a_N(q) s_N(q) .
 $$
We will apply Theorem \ref{thm:general-fibr} in the following situations:
\begin{itemize}
\item When the the action of $\pi_1(B)$ on $H_c^*(F)$ is trivial.
Then $R(Z)=e(F) T$, where $T$ is the trivial local system and (\cite[Proposition 2.4]{lomune})
$$
e(Z)=e(F) e(B).
$$
In particular, this happens when the fibration is trivial in the Zariski topology, or when
$Z$ is a $G$-space with isotropy $H<G$ such that $G/H\to Z\to B$ is a fibre bundle, 
and $G$ and $H$ ar e connected algebraic groups. Then 
$$
e(Z)=e(B)e(G)/e(H)
$$

\item When $B$ is one-dimensional: 
 $$
 F\too Z \too B=\CC-\{q_1,\ldots, q_\ell\}, 
 $$
Assuming that $B_\rho$ is a rational curve, we write
$R(Z)=a_1(q) T+ a_2(q) S_2+ \ldots + a_N(q)S_N$, where $T$ is the trivial representation
and $S_2, \ldots, S_N$ are the non-trivial representations. Then $e(T)=q-\ell$ and $e(S_i)=-(\ell-1)$,
$2 \leq i\leq N$. Hence
 \begin{equation}\label{eqn:onedimesional}
 e(Z)=(q- \ell) a_1(q) -(\ell -1) \sum_{i=2}^N a_i(q)=(q-1)\, e(F)^{inv} - (\ell-1) e(F),
 \end{equation}
where $e(F)^{inv}=a_1(q)$ is the E-polynomial of
the invariant part of the cohomology of $F$ and $e(F)=\sum_{i=1}^N a_i(q)$. (See \cite[Corollary 3]{mamu}.
\end{itemize}
An application of (\ref{eqn:onedimesional}) and Theorem \ref{thm:general-fibr} relevant to our situation is the following. Assume that we have a fibration $\oZ \to B=\CC- \lbrace 0 \pm 1 \rbrace$, with a $\ZZ_2$-action compatible with the action $\lambda \mapsto \lambda^{-1}$ in the base. We thus have a fibration
$$
\oZ/\ZZ_2 \rightarrow B'=\CC - \lbrace \pm 2 \rbrace
$$
Assume that the Hodge monodromy representation is of the form $R(\oZ/\ZZ_2)=aT+bS_{2}+S_{-2}+S_{0} \in R(\ZZ_2\times \ZZ_2)[q]$, where $T$ is the trivial representation, $S_{\pm 2}$ are trivial over $\pm 2$ and non-trivial over $\mp 2$ and $S_{0}=S_{2}\otimes S_{-2}$. Then we have
\begin{align}
R(\oZ) & = (a+d)T+(b+c)N \in R(\ZZ_2][q] \nonumber\\
e(\oZ) & =(q-3)(a+d)-2(b+c) \nonumber \\
e(\oZ)^{+} & = (q-2)a-(b+c+d) \label{eqn:eoX4Z2} \\
e(\oZ)^{-} & = (q-2)d-(a+b+c) \nonumber
\end{align}
\section{Basic blocks. Character varieties of curves of genus $g=1$} \label{sec:basicblocks}
Our first step is to compute the E-polynomials of what we call basic or building blocks. They correspond to the character varieties associated to complex curves of  genus $1$. They serve as a basis for the induction procedure performed in Chapter \ref{chapter:highgenus} to obtain the E-polynomials for arbitrary genus. For $\sldos$, these blocks were described in detail in \cite{lomune} and their E-polynomials were also given. For $g=1$, they are GIT quotients of the representation spaces
$$
X_{i}:=\lbrace (A,B)\in \sldos^{2} \mid [A,B]\sim C \rbrace
$$
where $C=\Id,-\Id,J_{+}, J_{-}, \xi_{\lambda}$ for $i=0,1,2,3,4$ respectively. We will also write $\oX_{i}= \lbrace (A,B)\in \sldos^{2} \mid [A,B]=C \rbrace$ for the fixed conjugacy class cases. The corresponding $\pgl$-representation space can be seen as the quotient by the torus action $(A,B)\sim(-A,B) \sim (A,-B)$, that is
$$
Z_{i}=\lbrace (A,B)\in \pgl^{2} \mid [A,B]=C \rbrace \sim X_{i}/(\ZZ_2\times \ZZ_2)
$$
where again $C=\Id,-\Id,J_{+},J_{-}, \xi_{\lambda}$ for $i=0,1,2,3,4,$ and also $\oZ_{i}=\oX_{i}/(\ZZ_2\times \ZZ_2)$.
They stratify $\pgl^{2}$ because if we take the $\pgl$-equivariant map
\begin{eqnarray*}
f : \pgl^{2} & \longrightarrow & SL(2,\mathbb{C}) \\
(A,B) & \longrightarrow & [A,B]
\end{eqnarray*}
then $Z_{i}=f^{-1}(C)$, for a conjugacy class $C$ so that $\pgl^{2}=\sqcup_{C\in \mathcal{C}} f^{-1}(C) = Z_{0} \sqcup \ldots \sqcup Z_{4}$. An analogous commutator map $\tilde{f}$ can be defined from $\sldos^{2}$ to $\sldos$, so that $\tilde{f}^{-1}(C)=X_{i}$. We proceed in what follows to describe and stratify each $Z_{i}$.

\subsection*{\boldmath{$Z_{0}=\lbrace (A,B)\in \pgl^{2} \mid [A,B]=\Id \rbrace$}}
First of all, we remove the cases when $A$ or $B$ are equal to $\Id$. In any of those cases, the other matrix can be an arbitrary matrix in $\pgl$, so
$$
Z_0=\lbrace \pgl \times [\Id]\rbrace \cup \lbrace [\Id]\times \pgl \rbrace \cup Z_0'
$$
where $Z_0'=\lbrace (A,B)\in Z_0 \mid [A],[B]\neq \Id \rbrace$. As a consequence, $e(Z_0)=2e(\pgl)-1+e(Z_0')$.
We can divide again $Z_0'$ into two subsets, $Z_0'=Z_0'^{a}\cup Z_0'^{b}$, where
$$
Z_0'^{a}=\lbrace (A,B)\in Z_0 : A,B\neq \Id : \tr(A)\neq \pm 2 \rbrace
$$
In this case, we know that $A$ is diagonalizable, and since $A$ and $B$ commute, both are diagonalizable with respect to the same basis and $(A,B) \sim \left( \begin{pmatrix} \alpha & 0 \\ 0 & \alpha^{-1} \end{pmatrix}, \begin{pmatrix} \beta & 0 \\ 0 & \beta^{-1} \end{pmatrix} \right)$, $\alpha,\beta  \neq 0,\pm 1$. The $\ZZ_{2}$-actions are $\alpha \mapsto -\alpha$, $\beta \mapsto -\beta$, so we may take $a:=\alpha^{2}, b:=\beta^{2}, a,b\neq 0,1$ as parameters for the $\pgl$-quotient.

If we write $\tilde{Y}_4= \lbrace (a,P) : a\in \CC - \lbrace 0, 1 \rbrace, P\in GL(2,\mathbb{C})/D \rbrace$ and $Y_4 = \lbrace A\in \pgl : \tr A \neq \pm 2 \rbrace$, we have the following diagram
$$
\xymatrix{\CC - \lbrace 0,1 \rbrace \ar[r]^{=} \ar[d] & \CC - \lbrace 0,1 \rbrace \ar[d] \\
\tilde{Y}_4\times \CC - \lbrace 0,1 \rbrace \ar[r]^{\phi} \ar[d]^{\pi_1} & Z_0^{'a} \ar[d]^{\pi_2} \\
\tilde{Y}_{4} \ar[r]^{\varphi} & Y_{4}}
$$
where $\pi_1$ is the projection onto the first factor and $\pi_2$ takes a pair $(A,B)\in Z_0^{'a}$ to $A\in Y_4$. The fibre is parametrized by the eigenvalue $b \in \CC - \lbrace 0,1 \rbrace$ associated to $B$. Besides, the horizontal map at the bottom takes a pair $(a,P) \in \tilde{Y}_{4}$ to $P\Big( \begin{smallmatrix} a & 0 \\ 0 & a^{-1} \end{smallmatrix} \Big)P^{-1}$, and the horizontal map in the middle takes an element $(a,P,b) \in \tilde{Y}_{4}\times \CC - \lbrace 0 \rbrace$ to the pair $\Big( P\Bigl( \begin{smallmatrix} a & 0 \\ 0 & a^{-1} \end{smallmatrix} \Bigr)P^{-1}, P\Bigl( \begin{smallmatrix} b & 0 \\ 0 & b^{-1} \end{smallmatrix} \Bigr)P^{-1} \Bigr)$.

Note that both maps $\phi,\varphi$ are generically 2:1, but for $a=-1$, which corresponds to $\lambda=\pm i$, $P_0$ belongs to the stabilizer of $A=\begin {pmatrix} \pm i & 0 \\ 0 & \mp i  \end{pmatrix}$. Leaving aside that case for the moment, we have a fibration:
$$
\xymatrix{\CC - \lbrace 0,1 \rbrace \ar[r]^{=} \ar[d] & \CC - \lbrace 0,1 \rbrace \ar[d] \\
\bar{\tilde{Y}}_4\times \CC \setminus \lbrace 0,1 \rbrace \ar[r]^{\phi} \ar[d]^{\pi_1} & Z_0^{''a} \ar[d]^{\pi_2} \\
\bar{\tilde{Y}}_{4} \ar[r]^{\varphi} & \bar{Y}_{4}}
$$
where in $Z_0^{''a}$ we are just removing the $a=-1$ case, so that $ Z_0^{''a}=Z_0^{'a} - \lbrace (A,B)\in Z_0 \mid \tr A=0 \rbrace $ and $
\bar{\tilde{Y}}_4=\CC - \lbrace 0,1,-1 \rbrace \times \pgl /D
$, $
\bar{Y}_4= \lbrace A \in \pgl \mid \tr A \neq \pm 2,0 \rbrace
$
The maps are 2:1, so
\begin{equation} \label{eqn:Z0''a}
e(Z_0^{''a})=e(\bar{\tilde{Y}}_4)^{+}e(\CC - \lbrace 0, 1 \rbrace)^{+}+ e(\bar{\tilde{Y}}_4)^{-}e(\CC - \lbrace 0, 1 \rbrace)^{-}
\end{equation}
The action takes $a$ to $a^{-1}$ on the fibre, so $
e(\CC - \lbrace 0,1 \rbrace)^{+}=e(\CC - \{ 2 \} )=q-1, \:
e(\CC - \lbrace 0,1 \rbrace)^{-}=-1
$. To compute $e(\tilde{\bar{Y}}_4)$, we use the fibration:
$$
\xymatrix{GL(2,\mathbb{C})/D \ar@{=}[r] \ar[d] & GL(2,\mathbb{C})/D \ar[d] \\ \tilde{\bar{Y}}_4 \ar[d]^{\tr} \ar[r] & \bar{Y}_4 \ar[d] \\ \CC \setminus \lbrace 0, 1,-1 \rbrace \ar[r] & \mathbb{C} \setminus \lbrace 0, 2 \rbrace}
$$

So $
e(\tilde{\bar{Y}}_4)^{+}=e(\tilde{\bar{Y}}_4/\mathbb{Z}_2)=e(\bar{Y}_4)
$, 
and:
\begin{eqnarray*}
e(\bar{Y}_4) & = &e(GL/D)^{+}e(\CC \setminus \lbrace 0,1,-1 \rbrace)^{+} +e(GL/D)^{-}e(\CC \setminus \lbrace 0,1,-1 \rbrace)^{-}\\
& = & q^{2}(q-2)+q(-1)\\
& = & q^{3}-2q^{2}-q
\end{eqnarray*}
We deduce from this fact that $ e(\tilde{\bar{Y}}_4)^{+}=q^{3}-2q^{2}-q$, $
e(\tilde{\bar{Y}}_4)^{-}=-2q
$, so going back to (\ref{eqn:Z0''a}):
\begin{eqnarray*}
e(Z_0^{''a}) & = & e(\bar{\tilde{Y}}_4)^{+}e(\CC -\lbrace 0, 1 \rbrace)^{+}+ e(\bar{\tilde{Y}}_4)^{-}e(\CC - \lbrace 0,1 \rbrace)^{-}\\
& = & (q^{3}-2q^{2}-q)(q-1)+(-2q)(-1)=q^{4}-3q^{3}+q^{2}+3q
\end{eqnarray*}

Now, let us compute the remaining fibre, when $a=-1$. In this case, $A=\big( \begin{smallmatrix} \pm i  & 0 \\ 0 & \mp i \end{smallmatrix} \big)$. Now, if $b \neq -1$, we have the diagram:
$$
\xymatrix{GL(2,\mathbb{C})/D \ar[d] \ar@{=}[r] & GL(2,\mathbb{C})/D \ar[d] \\ GL(2,\mathbb{C})/D \times \CC - \lbrace 0,1,-1 \rbrace \ar[r] \ar[d] & Z_0^{'''a} \ar[d]^{\underline{\tr}} \\ \CC - \lbrace 0,1,-1 \rbrace \ar[r] & \CC - \lbrace 0, 2 \rbrace}
$$
where again, the horizontal bottom map takes $a$ to $a + \frac{1}{a}$ and the horizontal middle map takes a pair $(a,P)$ to $(PAP^{-1}, P \Big( \begin{smallmatrix} a & 0 \\ 0 & a^{-1} \end{smallmatrix}\Big)P^{-1})$.
So again
\begin{eqnarray*}
e(Z_0^{'''a}) & = & e(GL(2,\mathbb{C})/D)^{+}e(\CC - \lbrace 0,1,-1 \rbrace)^{+} + e(GL(2,\mathbb{C})/D)^{-}e(\CC - \lbrace 0,1,-1 \rbrace)^{-}\\
& = & q^{3}-2q^{2}-q
\end{eqnarray*}
Finally, we take the case where $\tr A= \tr B=0$, i.e. $a=b=-1$. The pair $(A,B)$ can be simultaneously diagonalized so that
$$
(A,B) = \bigg( \begin{pmatrix} \pm i & 0 \\ 0 & \mp i \end{pmatrix}, \begin{pmatrix} \pm i & 0 \\ 0 & \mp i \end{pmatrix} \bigg)
$$
In this case, we have only to take account of the conjugacy orbit of this single element. The stabilizer is bigger than $D$, since $P_0$ also stabilizes the pair. $P_0$ is an element of order 2 in $\pgl$: it generates a subgroup isomorphic to $\mathbb{Z}_{2}$. So
$$
e(Z_0^{''''a})=e((\pgl/\lbrace D, P_0 \rbrace)=e((\pgl /D)/\mathbb{Z}_{2})=e(\pgl /D)^{+}=q^{2}
$$
Adding all up
$$
e(Z_0^{'a}) =  e(Z_0^{''a})+e(Z_0^{'''a})+e(Z_{0}^{''''a})= q^{4}-2q^{3} +2q
$$

Now, let
$$
Z_{0}'^{b}= \bigg\{ (A,B) \in Z_0 : (A,B) \sim \left( \begin{pmatrix} \lambda & b \\ 0 & \lambda^{-1} \end{pmatrix}, \begin{pmatrix} \mu & y \\ 0 & \mu^{-1} \end{pmatrix} \right), \lambda,\mu = \pm 1 \bigg\} $$
which, is isomorphic to
$$
Z_0'^{b} \cong X_0'^{b}/\mathbb{Z}_2 \times \mathbb{Z}_2
$$
$X_0'^{b}$ has four components, which correspond to $\lambda,\mu = \pm 1$, and the $\mathbb{Z}_2\times \mathbb{Z}_2$-action interchanges them.
So quotienting by $\mathbb{Z}_2\times\mathbb{Z}_2$ leaves us one component and therefore
$$
e(Z_0'^{b}) = (q^{2}-1)(q-1)
$$
(assuming $\lambda=\mu=1$, each orbit contains an element of the type $\left( \begin{pmatrix} 1 & 1 \\ 0 & 1 \end{pmatrix}, \begin{pmatrix} 1 & y \\ 0 & 1 \end{pmatrix} \right)$, we obtain the whole $Z_0'^{b}$ conjugating by elements in $GL(2,\mathbb{C})/U$)

So finally,
$$
e(Z_0')=2e(\pgl)-1 +e(Z_0^{a})+e(Z_0^{b})= (q^3-q)(q+1).
$$

\subsection*{\boldmath{$Z_1= \lbrace (A,B) \in \pgl^{2} \mid [A,B]=-\Id \rbrace $}}
Let $(A,B)\in SL(2,\mathbb{C})^{2}$ be such that $[A,B]=-\Id$, this forces that $\tr A=\tr B=0$. There exists a basis such that with respect to it,
$$
A=\begin{pmatrix} i & 0 \\ 0 & -i \end{pmatrix}, B= \begin{pmatrix} 0 & 1 \\ -1 & 0 \end{pmatrix}
$$
so that there a single $\pgl$-orbit, $X_{1}\cong \pgl$ and the moduli space is a single point. For the $\pgl$-case, we want to declare equivalent $(A,B)\sim (-A,B) \sim (A,-B) \sim (-A,-B)$ inside the $\pgl$-orbit, so let us consider
$$
G'= \lbrace P\in GL(2,\mathbb{C}) : PAP^{-1}=A, \; PBP^{-1}=-B \rbrace \cong \mathbb{C}^{\ast} =\bigg\{ \begin{pmatrix} a & 0 \\ 0 & -a \end{pmatrix} a\in \mathbb{C}^{\ast} \bigg\}
$$
$$
G''=\lbrace P\in GL(2,\mathbb{C}) : PAP^{-1}=-A, \; PBP^{-1}=B \rbrace \cong \mathbb{C}^{\ast} =\bigg\{ \begin{pmatrix} 0 & b \\ -b & 0 \end{pmatrix} b\in \mathbb{C}^{\ast} \bigg\}
$$

If we regard $G',G''$ projectively, each of these subgroups is just a point of order 2 in $\pgl$
$$
[G']=\left[ \begin{pmatrix} 1 & 0 \\ 0 & -1 \end{pmatrix} \right], \quad 
[G'']=\left[ \begin{pmatrix} 0 & 1 \\ -1 & 0 \end{pmatrix} \right].
$$
The subgroup of $\pgl$ generated by $G,G'$, which we will denote by $H$, is isomorphic to $\mathbb{Z}_{2} \times \mathbb{Z}_{2}$. We have that
$$
Z_1=X_1/\langle \tau_{1} \times \Id, \Id \times \tau_{2} \rangle \cong S \times \Big( \pgl/H \Big)
$$
Note that $\pgl$ acts by conjugation on the slice $P \rightarrow (PA_{0}P^{-1},PB_{0}P^{-1})$ and so does $H$, but the action of $H$ on the orbit, $\pgl$, is by translations $P\rightarrow PH$. All that remains to compute is $e(\pgl/H)$. Recall that when a finite subgroup $G$ of a Lie group acts by translation on a space $X$
\begin{eqnarray*}
r_g : X\times G & \longrightarrow & X \\
(x,g) & \longrightarrow & xg
\end{eqnarray*}
the action is homotopically trivial: we can build a path, for every $g\in G$, that connects $g$ and $\Id\in G$ and that allows us to build a homotopy between the given action and the trivial action.
We get from this that $e(H^{\bullet}(X/H))=e(H^{\bullet}(X))$. In our case, $H$ is a finite subgroup of the connected Lie group $\pgl$ acting on the space $X=\pgl$.
So $e(\pgl/H)=e(\pgl)$ and therefore
$$
e(Z_1)=e(\pgl/H)=e(\pgl)=q^{3}-q.
$$

\subsection*{\boldmath{$Z_{2}= \lbrace (A,B)\in \pgl^{2} \mid [A,B] \sim J_{+} \rbrace$}}

We consider
$$
Z_2=\lbrace (A,B)\in \pgl^{2} : [A,B]\sim J_{+} \rbrace, \quad
\overline{Z}_{2}=\lbrace (A,B)\in \pgl^{2} : [A,B]= J_{+} \rbrace
$$
which are quotients of the $\sldos$-representation varieties  $X_{2}$ and $\oX_{2}$ respectively.
Looking at the latter \cite{lomune}, the conditions imposed by the equation $[A,B]=J_{+}$ force $(A,B)$ to be of the form
$$
A= \begin {pmatrix} a & b \\ 0 & a^{-1} \end{pmatrix}, B=\begin{pmatrix} x & y \\ 0 & x^{-1} \end{pmatrix},
$$
where $(a,b,x,y)$ satisfy the equation
$$
yx(a^{2}-1) -ba(x^{2}-1)=1,
$$
which is compatible with the $\ZZ_{2}$-actions $(a,b,x,y)\sim (-a,-b,x,y)$ and $(a,b,x,y)\sim (a,b,-x,-y)$ that arise when we consider the $\pgl$-case. $\oX_{2}$ can be regarded as a family of lines $L_{(a,x)}$ parametrized by $(a,x)\in \CC^\ast \times \CC^{\ast} - \lbrace (\pm 1, \pm 1) \rbrace$, we look for the E-polynomial of the quotient under the $\ZZ_2$-actions. Since $e(L_{(a,x)})^{+}=e(L_{(a,x)})$, applying (\ref{eqn:+-}) repeatedly yields
$$
e(\overline{Z}_2)  = e(\CC)e( \mathbb{C}^{\ast} \times \mathbb{C}^{\ast} - \{(\pm 1, \pm 1) \} /\ZZ_{2}\times \ZZ_2)
 =  q((q-1)^{2}-1) =  q^{3}-2q^{2}
$$
and
$$
e(Z_2)=e(GL(2,\mathbb{C})/U)e(\overline{Z}_2)=(q^{2}-1)(q^{3}-2q^{2})=(q^3-q)(q^2-2q).
$$

\subsection*{\boldmath{$Z_3= \lbrace (A,B) \in \pgl^{2} \mid [A,B] \sim J_{-} \rbrace $}}
\label{strataz3}
If we write
$$
A=\begin{pmatrix} a & b \\ c & d \end{pmatrix} \qquad B=\begin{pmatrix} x & y \\ z & w \end{pmatrix}
$$
The equations $[A,B]=J_{-}, A,B \in \sldos$ are equivalent to the set of equations
\begin{eqnarray}
 z=2(x+w) \\ c=-2(a+d) \\ cy+2dw+bz=0 
\end{eqnarray}
where we consider the actions $(a,b,c,d)\sim (-a,-b,-c-d)$ and $(x,y,z,w)\sim (-x,-y,-z,-w)$ for the $\pgl$-case.
We may subdivide $Z_3,\overline{Z}_3$ in the following substrata, following the $\sldos$-case \cite{lomune}:
\subsubsection*{Case 1: $\tr B=0$} This case provides a space $\overline{X}_{3}'$ isomorphic to $\CC^{\ast} \times \lbrace \pm i \rbrace$, where $\ZZ_{2}\times \ZZ_2$ acts by identifying both components and takes $a\in \CC^{\ast}$ to $-a\in \CC^{\ast}$, so $e(\overline{Z}_{3}')=e(\CC^{\ast} \times \lbrace \pm i \rbrace /\ZZ_2)e(Stab(J_{+}))=(q-1)q$ and $e(Z_{3}')=(q-1)(q^{3}-q)$.

\subsubsection*{Case 2: $\tr B \neq 0$}
In this case, conjugating by a suitable matrix sets $w=0, z=2x$ and
$$
A=\begin{pmatrix} a & \frac{-(a+d)}{2x^{2}} \\ -2(a+d) & d \end{pmatrix}, \quad 
B=\begin{pmatrix} x & -1/2x \\ 2x & 0 \end{pmatrix}.
$$
Since $\det A=1$, we arrive to the equation
$$
ad-\frac{(a+d)^2}{x^{2}}=1
$$
We stratify again in three subcases
\subsubsection*{$S^{a} : a+d=0$}
In this case, we get that $a=-d= \pm i$ , so $A$ is fixed and $B$ is parametrized by taking $x\in \CC^{\ast}$, the $\ZZ_2$-actions take $x\to -x$ and $i\to -i$. So $
e(S^{a})=e(\CC^{\ast}/\ZZ_2)=q-1.$
\subsubsection*{$S^{b},S^{c}: a+d \neq 0$}
In this case, we may take the parameter $u=\frac{(a+d)^{2}}{x^{2}}$, invariant under the $\ZZ_2$-actions, so that the resulting equation is $ad=1+u$. We may subdivide again in two cases according to the values of $u$:
\begin{itemize}
\item[-] $u= -1$. The resulting equation is $ad=0, (a,d)\neq(0,0)$. We get 
$$
e(S^{b})=e(\{ad=0\}/\ZZ_2)=2e(\CC^{\ast})=2(q-1).
$$
\item[-] $u\neq 0,-1$ Now, for every $u \in \mathbb{C} - \lbrace 0,-1 \rbrace$, $ad=1+u$, $(a,d)\sim (-a,-d)$. If we complete the set of solutions projectively we get a conic for every $u$, which gives a contribution of $(q+1)(q-2)$ ($u\neq 0,1$). We still have to remove the points at infinity that contribute $2(q-2)$ and the points which satisfy $a+d=0$, which gives us the equation $-a^{2}=1+u$, again a conic with three points removed: it has polynomial $q-2$. We get that
$$
e(S^{c})=(q+1)(q-2)-2(q-2)-(q-2)=(q-2)^{2}
$$
\end{itemize}
Adding all up
$$
e(S)=e(S^{a})+e(S^{b})+e(S^{c})=q^{2}-q+1
$$
So
$$
e(\overline{Z}_{3}'')=q(q^{2}-q+1), \quad e(Z_{3}'')=q(q^{2}-1)(q^{2}-q+1).
$$
Finally,
$$
e(\overline{Z}_{3})=q^{3}, \quad e(Z_{3})=e(Z_{3}')+e(Z_{3}'')=(q^{3}-q)q^{2}.
$$

\subsection*{\boldmath{$Z_4=\lbrace (A,B) \in \pgl^{2} \mid [A,B]\sim \xi \rbrace$}}
\label{strataz4}
Finally, we compute the E-polynomial of the last stratum,
$$
Z_4=\lbrace (A,B) \in \pgl^{2} \mid [A,B] \sim \begin{pmatrix} \lambda & 0 \\ 0 & \lambda^{-1}\end{pmatrix} \lambda\neq 0 \pm 1 \rbrace.
$$
We also define, for $A,B\in \pgl$
$$
\overline{Z}_4=\lbrace (A,B, \lambda) \mid [A,B] = \begin{pmatrix} \lambda & 0 \\ 0 & \lambda^{-1}\end{pmatrix} \lambda\neq 0 \pm 1 \rbrace, \quad \overline{Z}_{4,\lambda}=\lbrace (A,B) \mid [A,B] = \begin{pmatrix} \lambda & 0 \\ 0 & \lambda^{-1}\end{pmatrix} \rbrace.
$$
Analogous definitions were given in \cite{lomune} for the $\sldos$-representation spaces $X_{i},\overline{X}_{i}, \oX_{4,\lambda}$. Like we did before, we can compute the $E$-polynomials of $Z_{4}$ and $\overline{Z}_{4}$ regarding these spaces as quotients  of $X_{4}$ and $\overline{X}_{4}$ by the torus action, which is a $\mathbb{Z}_{2}\times\mathbb{Z}_{2}$-action. We stratify $X_{4}$ closely following \cite{lomune}, obtaining slices $S_{i}$ for the conjugacy action and analyzing the $\ZZ_2$-actions in each case. For each of the strata $\overline{X}^{i}_{4}$, using that the stabilizer of $\xi$ in $\pgl$ is isomorphic to $\mathbb{C}^{\ast}$,
$$
\overline{X}_{4}^{i}\cong \mathbb{C}^{\ast}\times S^{i} \quad X^{i}_{4}\cong (\pgl \times S^{i})/\mathbb{Z}_2
$$
where the latter $\ZZ_2$-action is given by conjugation by $P_{0}$. For $Z_{4}^{i}$ and $\overline{Z}_{4}^{i}$, since the $\ZZ_2$-actions commute with the conjugacy action, we will just need to find the quotients of each slice, so that 
$$
\overline{Z}_{4}^{i}\cong \mathbb{C}^{\ast}\times (S^{i}/(\ZZ_2\times \ZZ_2)) \quad X^{i}_{4}\cong (\pgl \times (S^{i}/(\ZZ_2 \times \ZZ_2)))/\mathbb{Z}_2
$$
For $(A,B)\in \overline{X}_{4,\lambda}$, $\tr B=\tr \:\xi B$, and we also have that $\tr A = \tr \:\xi^{-1}A$. Therefore
$$
A= \begin{pmatrix} a & b \\ c & \lambda^{-1}a \end{pmatrix} \quad B=\begin{pmatrix} x & y \\ z & \lambda x\end{pmatrix}
$$
Finally, the equality $AB=\xi BA$ and the equations arising from the $\det A= \det B =1$, give us
\begin{eqnarray}
\lambda^{-1}a^{2}-bc &=& 1 \label{eqn:x4lam1} \\
\lambda x^{2}-yz &=& 1 \label{eqn:x4lam2}\\
ax+bz &=& \lambda(ax+cy) \label{eqn:x4lam3}
\end{eqnarray}
The $\ZZ_2$-actions are $(a,b,c)\sim (-a,-b,-c)$ and $(x,y,z)\sim (-x,-y,-z)$. Note that equations (\ref{eqn:x4lam1}), (\ref{eqn:x4lam2}), (\ref{eqn:x4lam3}) are compatible with them. We will stratify $Z_{4}$ according to different values of $a,b,c,x,y$ and $z$.

\subsubsection*{\textbf{Hodge monodromy of }\boldmath{$\overline{Z}_{4}$}}
Along with the stratification of $\overline{Z}_{4}$, we will study the Hodge monodromy representation (\ref{eqn:Hodge-mon-rep}) of the fibration
$$
\pi: \overline{Z}_{4}\longrightarrow \mathbb{C} \setminus \lbrace 0, \pm 1 \rbrace,
$$
where $(A,B,\lambda)$ is mapped to $\lambda$. The fibres are the spaces $\overline{Z}_{4,\lambda}$, diffeomorphic and of balanced type. We take loops $\gamma_{-1},\gamma_{0}, \gamma_{1}$ around $ -1,0,1 $ respectively, our aim is to compute the Hodge monodromy representation
$$
R(\overline{Z}_{4})\in R(\Gamma)[q]
$$
where $\Gamma = \langle \gamma_{-1},\gamma_{0},\gamma_{1} \rangle \cong H_{1}(B)$. We will proceed computing $R(\overline{Z}_{4}^{i})$ in each stratum $\overline{Z}_{4}^{i}$. Since $\overline{Z}_{4,\lambda}^{i}\cong S_{\lambda}^{i}\times \CC^{\ast}$, we will have that $R(\overline{Z}_{4}^{i})=(q-1)R(S^{i})$.

\subsubsection*{\boldmath{$Z_4^{0}=\lbrace b=c=0 \rbrace$}}
If $(A,B)\in X_4$ and $b=c=0$, the resulting equations are:
$$
a^{2} = \lambda, \quad
x = 0, \quad
yz =-1 
$$
This makes the set of pairs $(A,B)\in Z^{0}_{4}$ to be of the form:				
$$
\begin{pmatrix} a & 0 \\ 0 & a^{-1} \end{pmatrix} \quad \begin{pmatrix} 0 & y \\ -\frac{1}{y} & 0 \end{pmatrix} 
$$
where $y\in \CC^{\ast}, a\in \CC - \lbrace 0,\pm i, \pm 1 \rbrace , y\sim -y, a\sim -a$. The slice $S_{4}^{0} $ is just a single point. We deduce that
$e(\overline{Z}_{4,\lambda}^{0})=(q-1), \: e(\overline{Z}_{4}^{0})=(q-3)(q-1)$ and also
$$
R(\overline{S}_{4}^{0})=T
$$

To compute $Z_{4}^{0}$, note that $\overline{Z^{0}_{4}}$ gives us a slice for the action by conjugation of $\pgl /D$:
\begin{eqnarray*}
\overline{Z_{4}^{0}}\times \pgl /D & \longrightarrow & Z_{4}^{0} \\
((A,B,\lambda),P) & \longrightarrow & (P^{-1}AP,P^{-1}BP)
\end{eqnarray*}
so that, taking into account the action of $\ZZ_2$ on $\overline{Z_{4}^{0}}$ given by permutation of the eigenvalues, 
$$
e(Z_4^{0})=e(\overline{Z_{4}^{0}})^{+}e(\pgl/D)^{+}+e(\overline{Z_{4}^{0}})^{-}e(\pgl /D)^{-}
$$

To compute $e(\overline{Z_4^{0}})^{+}$, since $\overline{Z_{4}^{0}}$ is parametrized by pairs $(a,y), a \sim -a, y\sim -y$, we may use as parameters $A=a^{2}\in \CC-\lbrace 0,1,-1 \rbrace,Y=y^{2} \in \CC^{\ast}$. The $\mathbb{Z}_{2}$-action takes $A$ to $A^{-1}$ and $Y$ to $\frac{-1}{Y}$, so the quotient is parametrized by $A+\frac{1}{A} \in \CC - \lbrace 0,2 \rbrace$ and $Y-\frac{1}{Y}\in \CC$.

So:
$$
e(\overline{Z^{0}_{4}}/\mathbb{Z}_{2})=(q-2)q+1 \quad 
e(\overline{Z^{0}_{4}})^{-}=e(\overline{Z^{0}_{4}})-e(\overline{Z^{0}_{4}})^{+}=-2q+2
$$
and finally:
$$
e(Z_{4}^{0})=q^{2}(q^{2}-2q+1)+q(2-2q)=(q^{3}-q)(q-2)
$$
\subsubsection*{\boldmath{$Z_4^{1}=\lbrace x=y=0 \rbrace$}}
This case is analogous to the previous one, we get
$$
R(S_{4}^{1})=T, \quad e(\overline{Z}^{1}_{4}/\mathbb{Z}_{2})=(q-2)q+1, \quad 
e(Z_{4}^{1})=(q^{3}-q)(q-2)
$$
\subsubsection*{\boldmath{$Z_4^{2}= \lbrace b=0, c\neq 0 \rbrace \cup \lbrace b\neq 0, c=0 \rbrace$}}
Let us consider the case
$$
\overline{Z_{4}^{2}}= \lbrace (A,B)\in Z_4 \mid b=0, c\neq 0 \rbrace \cup \lbrace (A,B)\in Z_4 \mid b\neq 0, c=0 \rbrace
$$
First, we can rescale so that $c=1$. $\overline{Z_{4}^{2}}$ has two strata, the equations for the first one are
$$
a^{2} =  \lambda, \quad
a^{2}x^{2}-yz  =  1, \quad
(1-\lambda)ax = \lambda y
$$
From the last two equations, we deduce that
$$
y= \frac{(1-\lambda)ax}{\lambda}, \quad
z= \frac{a^{2}x^{2}-1}{y}
$$
so that $\overline{Z_{4}^{2}}$ is parametrized by pairs $(a,x)$ where $a\in \CC - \lbrace 0,1,i \rbrace, x\in \CC^{\ast}, x\sim -x, a\sim -a$. We denote this slice by $S_{4}^{2}$. When we look at the monodromy of the slice as we go around the origin, $a$ gets swapped to $-a$, but they are already identified and the monodromy is trivial. There are two isomorphic components, so
$$
R(S_{4}^{2})=2(q-1)T
$$
and if we  multiply by the stabilizer, we get that
$$
e(\overline{Z_{4,\lambda}^{2}})=2(q-1)^2 \quad
e(\overline{Z_{4}^{2}})=2(q-3)(q-1)(q-1)
$$
To compute $Z_{4}^{2}$, note that the $\ZZ_2$-action interchanges the two components. Therefore
$ Z^{2}_{4} \cong \pgl \times S^{2}_{4} $ and 
$$
e(Z^{2}_{4})=(q^{3}-q)(q-1)(q-3).
$$
\subsubsection*{\boldmath{$Z_4^{3}=\lbrace x=0, y\neq 0 \rbrace \cup \lbrace x\neq 0, y=0 \rbrace$ }}
Similarly,  
$$
R(S_{4}^{3})=2(q-1)T, \quad e(\overline{Z_{4}^{3}})=2(q-3)(q-1)(q-1), \quad e(Z^{3}_{4})=(q^{3}-q)(q-1)(q-3)
$$
\subsubsection*{\boldmath{$Z_4^{4}=\lbrace b=z=0 \rbrace \cup \lbrace c=y=0 \rbrace$}}
This component is given as the intersection of $Z_{4}^{2}$ and $Z_{4}^{3}$, since it is impossible that $b=0$ and $y=0$ (or $c=0,z=0$) simultaneously.
The conditions $b=z=0$ and the slice fixing condition $c=1$ yield the set of equations:
$$
a^{2} =  \lambda, \quad
x^{2} = \frac{1}{\lambda}, \quad
(\lambda-1)ax + y\lambda  =  0 
$$
where again $x\sim -x, a\sim -a$. We point out that $a,x$ are both fully determined each value of $\lambda$, as well as $y$. Therefore, the slice consists of two points and it is parametrized by $\lambda \in \mathbb{C} \setminus \lbrace 0, \pm 1 \rbrace$. There is no monodromy around the origin because of the identifications $x\sim \-x$ and $a \sim -a$. We get that $R(S_{4}^{4})=2T$ and $e(\overline{Z_{4,\lambda}^{4}})=2(q-1)$.
Recalling that again we have two isomorphic components,we get:
$$
e(\overline{Z^{4}_{4}})=2(q-1)(q-3)
$$
The $\mathbb{Z}_{2}$ action interchanges them, so $Z_{4}^{4}\cong \mathbb{C} \setminus \lbrace 0, \pm 1 \rbrace \times \pgl$ and $e(Z_4^{4})=(q^{3}-q)(q-3)$.
\subsubsection*{\boldmath{$Z_4^{5}=\lbrace b,c,y,z \neq 0 \rbrace$}}
Equations (\ref{eqn:x4lam1}),(\ref{eqn:x4lam2}), (\ref{eqn:x4lam3}), when we fix $b=1$, become
\begin{eqnarray*}
\lambda^{-1}a^{2}-bc & = & 1 \\
\lambda x^{2}-yz & = & 1 \\
ax+bz & = & \lambda (ax+cy)
\end{eqnarray*}
where
$$
A= \begin{pmatrix} a & 1 \\ c & \lambda^{-1}a \end{pmatrix}, \quad B= \begin{pmatrix} x & y \\ z & \lambda x \end{pmatrix}.
$$
The action of $\mathbb{Z}_{2} \times \mathbb{Z}_{2}$, generated by $\phi, \varphi$, where $\phi((A,B))=(-A,B), \: \varphi((A,B))=(A,-B)$ plus the slice fixing condition is now the following
$$
A= \begin{pmatrix} a & 1 \\ c & \lambda^{-1}a \end{pmatrix} \overset{\phi}{\longrightarrow} -A= \begin{pmatrix} -a & -1 \\ -c & -\lambda^{-1}a \end{pmatrix} \overset{\text{slice}}{\longrightarrow} \begin{pmatrix} -a & 1 \\ c & -\lambda^{-1}a \end{pmatrix}
$$
$$
\begin{pmatrix} x & y \\ z & \lambda x \end{pmatrix} \overset{\text{slice}}{\longrightarrow} \begin{pmatrix} x & -y \\ -z & \lambda x \end{pmatrix}
$$
and:
$$
\begin{pmatrix} x & y \\ z & \lambda x \end{pmatrix} \underset{\varphi}{\longrightarrow} \begin{pmatrix} -x & -y \\ -z & -\lambda x \end{pmatrix}
$$
Combining the three equations as it was done in [1], we see that $\overline{Z_{4}^{5}}$ is given by the single equation
\begin{equation}
\lambda x^{2}+a(1-\lambda)xy+(\lambda -a^{2})y^{2}=1 \label{eqn:z4lamconic}
\end{equation}
under the equivalence relation $(a,x,y)\sim^{\phi}(-a,x,-y)$ and $(a,x,y)\sim^{\varphi}(a,-x,-y)$. These are quotients of affine conics over the plane $(x,y)$ parametrized by $(a,\lambda)$, $a^{2}\neq \lambda$ (this condition is equivalent to $c\neq 0$). The conics have discriminant
$$
D=((\lambda + 1)a-2\lambda)((\lambda + 1)a + 2\lambda)
$$
The extra $\mathbb{Z}_{2}$-action given by permutation of the eigenvalues, considering the slice fixing condition $b=1$ is the following \cite{lomune}
\begin{eqnarray*}
a & \rightarrow & \lambda^{-1}a \\
c & \rightarrow & c \\
x & \rightarrow & \lambda x \\
y & \rightarrow & \frac{z}{c}= \lambda y + \frac{\lambda(\lambda -1)a}{a^{2}-\lambda}x \\
z & \rightarrow & yc
\end{eqnarray*}
We divide the study of this space depending on whether $D=0$ or not (cases $Z_{4}^{5}$ and $Z_{4}^{6},Z_{4}^{7},Z_{4}^{8}$ respectively).
When $D=0$, then:
$$
a=\pm \frac{2\lambda}{\lambda + 1}
$$
Note that the action of $\phi$ interchanges both components. Therefore, quotienting by $\phi$ leaves us the value $a=\frac{2\lambda}{\lambda + 1}$. In this case, the equation of the family of conics is
$$
\bigg( x-\frac{\lambda - 1}{\lambda +1}y \bigg)^{2}=\frac{1}{\lambda}
$$
under the equivalence relation $(x,y)\sim^{\varphi}(-x,-y)$. Taking $\mu := x-\frac{\lambda +1}{\lambda -1}y$ so that $\mu^{2}=\lambda^{-1}$, the action of $\varphi$ takes $\mu$ to $-\mu$, so the new coordinates are pairs $(x,\mu)$ with the relation $(x,\mu)\sim(-x,-\mu)$, where $\mu\neq \lbrace 0, \pm 1, \pm i \rbrace$ and $x\neq \pm \mu$. We may change variables again taking  $X:=\frac{x}{\mu}$, so that we get a bijection between pairs $(x,\mu)$ and pairs $(X,\mu)$. Now, the $\varphi$- action leaves $X$ invariant ($(X,\mu)\sim (X,-\mu)$), and the slice is parametrized by pairs $(X,\mu^{2})$ where $X\neq \pm 1$ and $\mu^{2} \neq \lbrace 0, 1,-1 \rbrace$. The monodromy around the origin is trivial: takes $\mu$ to $-\mu$, which are already identified by $\varphi$. Its Hodge monodromy representation is
$$
R(S_{\lambda,4}^{5})=(q-2)T
$$
and also
$$
e(\overline{Z_{4,\lambda}^{5}})=(q-1)(q-2), \quad 
e(\overline{Z_{4}^{5}})=(q-1)(q-2)(q-3)
$$
To compute $Z_{4}^{5}$, the $\mathbb{Z}_{2}$-action takes $x$ to $\mu^{-2}x$ and takes $\mu$ to $-\mu^{-1}$. Therefore, it takes $X$ to $-X$ and $\mu^{2}$ to $\mu^{-2}$. The set of pairs $(X,\mu^{2})$ under this equivalence relation has polynomial
$$
e(S^{5}/\mathbb{Z}_{2})=(q-1)(q-2)+1.
$$
So
$$
e(Z_{4}^{5})=(q^{3}-q)((q-1)(q-2)+1)
$$
\subsubsection*{\boldmath{$Z_4^{6}=\lbrace a,c,x,y,D\neq 0 \rbrace$}}
We deal now with the case where the discriminant is non-zero. Recall equation (\ref{eqn:z4lamconic}):
$$
\lambda x^{2}+a(1-\lambda)xy+(\lambda- a^{2})y^{2}=1,
$$
where the $\mathbb{Z}_{2}\times\mathbb{Z}_{2}$-action identifies $(a,x,y)\sim^{\phi}(-a,x,-y)$ and $(a,x,y)\sim^{\varphi}(a,-x,-y)$. When $D\neq 0$, if we add up the points at infinity we get over each possible $(a,\lambda)$ before quotienting a conic, $\mathbb{P}^{1}$, thus obtaining a conic bundle $\mathbb{P}^{1} \longrightarrow E \longrightarrow \lbrace (a,\lambda) \mid a^{2}\neq \lambda, a\neq \pm 2\lambda/\lambda+1  \rbrace$.

Now $\phi$ takes $E_{(a,\lambda)}$ to $E_{(-a,\lambda)}$, so it takes fibres to fibres and therefore quotienting by $\phi$ we get a conic bundle $\mathbb{P}^{1}\longrightarrow E/\phi \longrightarrow \lbrace (a^{2},\lambda) \mid a^{2}\neq \lambda, 4\lambda^{2}/(\lambda+1)^{2}\rbrace$. (the $\phi$-action identifies both values $a$ and $-a$). $\varphi$ makes identifications in the fibre over $(a^{2},\lambda)$: it takes $(x,y)\in E_{(a,\lambda)}$ to $(-x,-y)\in E_{(a,\lambda)}$. We get
$$
\mathbb{P}^{1}/\varphi \longrightarrow E/\mathbb{Z}_{2}\times \mathbb{Z}_{2} \longrightarrow \lbrace (a^{2},\lambda) \rbrace
$$
The quotient $\mathbb{P}^{1}/\varphi$ is isomorphic to $\mathbb{P}^{1}$, with fixed points at infinity (there are two of them). The base is parametrized by pairs $ (a^{2},\lambda)$ satisfying $\lambda \neq 0,\pm 1, a^{2}\neq \lambda, \frac{4\lambda^{2}}{(\lambda+1)^{2}}$, and has polynomial $(q-3)(q-1)-(q-3)=(q-3)(q-2)$. Therefore
$$
e(S_{4}^{6})=e(B)e(\mathbb{P}^{1})=(q+1)(q-3)(q-2)
$$
and consequently $e(\overline{Z_{4,\lambda}^{6}})=e(Stab)e(S_{4,\lambda}^{6})=(q-1)(q-2)(q+1)$ and also $e(\overline{Z_{4}^{6}})=e(Stab)e(S_{4}^{6})=(q-1)(q+1)(q-2)(q-3).$ There is no monodromy: going around $\lambda=0$ does  neither affect the conic fibration or the excluded values for the base. We also get
$$
R(S_{\lambda}^{6})=(q+1)(q-2)T
$$

To compute $Z_{4}^{6}$, we have to take account of the $\mathbb{Z}_{2}$-action given by the permutation of the eigenvalues. In terms of our stratification, it was checked in \cite{lomune} that the action takes
$$
(a,\lambda)\mapsto(\lambda^{-1}a, \lambda^{-1})
$$
and takes the element $(x,y)\in E_{(a,\lambda)}$ to $(\lambda x, \lambda y + \frac{\lambda(\lambda -1)a}{a^{2}-\lambda}x)\in E_{\lambda^{-1}a,\lambda^{-1}}$. Besides, the $\mathbb{Z}_{2}$-action is compatible with the actions given by $\phi$ and $\varphi$. Since the action takes fibres to fibres, our aim is to compute $B/\mathbb{Z}_{2}$. If we consider $A=\frac{a^{2}}{\lambda}$, note that the $\mathbb{Z}_{2}$-action takes $A\longrightarrow \frac{\lambda a^{2}}{\lambda^{2}}=A$, so $A$ remains invariant and there is a bijection between pairs $(A,\lambda)$ and pairs $(a^{2},\lambda)$. The action is now $(A,\lambda)\sim (A,-\lambda)$, so we consider the parameter $s=\lambda+\lambda^{-1}$. Hence the quotient $B/\mathbb{Z}_{2}$ is parametrized by pairs $(A,s)$ where
\begin{itemize}
\item $\lambda \neq 0,\pm 1$ implies that $s \neq \pm 2$
\item $a^{2}\neq \frac{4\lambda^{2}}{(\lambda + 1)^{2}}$ implies that $A\neq \frac{4}{2+s}$.
\item $a^{2}\neq \lambda$ implies that $A\neq 1$.
\end{itemize}
which is isomorphic to $\mathbb{C} - \lbrace 1 \rbrace \times \mathbb{C} - \lbrace \pm 2 \rbrace$. We still have to remove the subset $A=\frac{4}{2+s}$, isomorphic to $\mathbb{C} - \lbrace 2,-2 \rbrace$. We can conclude that $e(B/\mathbb{Z}_{2})=(q-1)(q-2)-(q-2)=(q-2)^{2}$ and deduce from this fact that
$$
e(Z_{4}^{6})=e(B/\mathbb{Z}_{2})e(\mathbb{P}^{1})e(\pgl)=(q^{3}-q)(q+1)(q-2)^{2}
$$

\subsubsection*{\boldmath{$Z_4^{7}=\lbrace a,c,x,y, D\neq 0, yz=0 \rbrace $}}
We have to remove the points from the case $yz=0$. To equations (\ref{eqn:x4lam1}), (\ref{eqn:x4lam2}), (\ref{eqn:x4lam3}) we need to add the condition $yz=0$, which is equivalent to $x^{2}=\frac{1}{\lambda}$.
From equation (\ref{eqn:z4lamconic}), this forces $y=0$ or $y=\frac{(\lambda -1)ax}{\lambda -a^{2}}$. We recall that we also have to consider the $\mathbb{Z}_{2}\times \mathbb{Z}_{2}$-action given by $\phi,\varphi$.
Taking this into account, we distinguish two cases:
\begin{itemize}
\item $a=0$. In this case we get a unique value for $y$, $y=0$. Also both values of $x$ are identified by $\varphi$; the action of $\phi$ is trivial since $a=0,y=0$. For fixed $\lambda$ this is just a point and when $\lambda$ varies we get a contribution of $(q-3)$.
\item $a\neq 0$. In this case there are two different values of $y$. For each $(a,\lambda)$, there are two possible values of $x$ which give rise to two possible values of $y$, i.e, four points in each fibre.	Now $\varphi$ identifies these points in pairs, identifying $(x,0)\sim (-x,0)$ and $(x, \frac{(\lambda -1)ax}{\lambda -a^{2}}) \sim (-x, \frac{(\lambda -1)a-x}{\lambda -a^{2}}) $. Besides, the $\phi$-action identifies the points over $(a,\lambda)$ and $(-a,\lambda)$.
\end{itemize}

We can take $a^{2}$ as a parameter for base, with three excluded values. There are two components (corresponding to $y=0$ and $y\neq 0$), so it has polynomial $e(S^{7}_{\lambda})=2(q-3)+1$.
The total slice has E-polynomial $e(S_{4}^{7})=2(q-3)^{2}+(q-3)$. We get:
$$
e(\overline{Z_{4,\lambda}^{7}})=(q-1)(2q-5), \quad
e(\overline{Z_{4}^{7}})=(q-1)(2(q-3)^{2}+(q-3))
$$
The monodromy around $\gamma_{0}$ takes $x$ to $-x$ and therefore swaps the two non-zero values of $y$, but so does the $\varphi$-action too and the monodromy is therefore trivial
$$
R(S_{\lambda,4}^{7})=(2(q-3)+1)T=(2q-5)T
$$

To compute $Z_{4}^{7}$, we have to consider the $\mathbb{Z}_{2}$-action.
It takes $a \longrightarrow\lambda^{-1}a$, $\lambda \longrightarrow \lambda^{-1}$, $x\longrightarrow \lambda x=x^{-1}$ and $ 0 \longrightarrow \frac{\lambda(\lambda -1)ax}{a^{2}-\lambda}=\frac{x^{2}a(1-\lambda^{-1})x^{-1}}{a^{2}x^{4}-\lambda^{-1}}.
$
Note that this last value (which we denote by $y'$) is the non-zero $y$-value over $(\lambda^{-1}a,\lambda^{-1})$.
In other words, the $\mathbb{Z}_{2}$-action takes the value $(x,0)$ over $(a,\lambda)$ to $(x^{-1},y')$ over $(\lambda^{-1}a,\lambda^{-1})$. It also takes the non-zero $y$-value $(x,y)$ over $(a,\lambda)$ to $(x^{-1},0)$ over $(\lambda^{-1}a,\lambda^{-1})$. So the $\mathbb{Z}_{2}$-action identifies the two components given by the values which satisfy $y\neq 0$ and the values which satisfy $y=0$: we get rid of this $\mathbb{Z}_{2}$-action by considering just one component. The contribution from $a=0$ is $q-2$ and the contribution from $a\neq 0$ is $(q-3)^{2}$. Hence:
$$
e(Z_{4}^{7})=e(B/\mathbb{Z}_{2})e(\pgl)=((q-3)^2+(q-2))(q^{3}-q)
$$

\subsubsection*{\boldmath{$Z_4^{8}= \lbrace a,c,x,y,D\neq 0, \textbf{points at infinity} \rbrace$}}
Finally, we remove the points at infinity of the $\mathbb{P}^{1}$-bundle. The equation for these points is
$$
(\lambda -a^{2})y^{2}+a(1-\lambda)xy+\lambda x^{2}=0
$$
Taking $Y=\frac{y}{x}$ as a parameter,
$$
\bigg( Y+\frac{1}{2}\frac{a(1-\lambda)}{\lambda -a^{2}} \bigg)^{2}=\frac{D}{4(\lambda-a^{2})^{2}}
$$
where $(a,Y)\sim (-a,-Y)$. Now, writing $\alpha=Y+\frac{1}{2}\frac{a(1-\lambda)}{\lambda-a^{2}}$, we get the equation $ \alpha^{2}=\frac{D}{4(\lambda-a^{2})^{2}}$. Since $D=a^{2}(1+\lambda)^{2}-4\lambda^{2}$, defining $A:=\frac{(1+\lambda)a}{\lambda}$ and $B:=\frac{2\alpha(\lambda-a^{2})}{\lambda}$
we finally arrive at the equation
$$
B^{2}=A^{2}-4
$$
where $(A,B)\sim (-A,-B)$ by the $\phi$-action. Notice that $A,B$ are for every $\lambda$ in bijection with pairs $(a,Y)$. In addition, $\lambda \neq 0,\pm 1$, $A\neq \pm 2 (D\neq 0)$ and $A^{2}\neq \frac{(1+\lambda)^{2}}{\lambda} (a^{2}\neq \lambda)$.

We get a conic for fixed $\lambda$ despite of the $\mathbb{Z}_{2}$-action (its $\phi$-quotient is again isomorphic to $\mathbb{P}^{1}$) with 3 points removed (two at infinity and $(2,0)\sim (-2,0)$). Varying $\lambda$ it gives us a contribution of $(q-3)(q-2)$. Besides, we still have to remove the points corresponding to $A^{2}=\frac{(1+\lambda)^{2}}{\lambda}$. Substituting into the equation, we get that $A=\pm \frac{(1+\lambda)}{(1-\lambda)}B$. So for each $\lambda$, we get four points $(A,B)$ identified by pairs ($(A,B)\sim (-A,-B)$) that need to be removed. It gives a contribution of $2(q-3)$, so:
$$
e(S_{4}^{8})=(q-3)(q-2)-2(q-3)=(q-3)(q-4)
$$
and therefore $e(\overline{Z_{4,\lambda}^{8}})=(q-1)(q-4)$ and $e(\overline{Z_{4}^{8}})=(q-1)(q-3)(q-4)$. The monodromy action is trivial again (for $\sldos$, it swaps the missing points, which are already identified in the $\pgl$-case) so
$$
R(S_{4}^{8})=(q-4)T
$$

Finally, the $\mathbb{Z}_{2}$-action leaves $A$ and $B$ invariant and sends $\lambda$ to $\lambda^{-1}$, so  we can take $s=\lambda+\lambda^{-1}$ as a parameter. We have again conics with three points deleted and two excluded values for $s$, namely $s=\pm 2$. This gives a contribution of $(q-2)(q-2)$.
In addition, we have to remove the subvariety defined by $A^{2}=\frac{(1+\lambda)^{2}}{\lambda}$, i.e. $A^{2}=s+2$.
Since $s\neq \pm 2$, we get that $A\neq 0, \pm 2$ and therefore $(A,B)\neq (0,\pm 2i),(\pm 2,0)$. The subvariety is a conic $B^{2}=A^{2}-4$ with the two points at infinity and the two points $(0,2i),(2,0)$ removed, giving a contribution of $q-3$. Therefore
$$
e(Z_{4}^{8})=(q^{3}-q)((q-2)(q-2)-(q-3))
$$

So, adding all up:
\begin{eqnarray*}
e(Z_{4}) & = & e(Z_{4}^{0})+e(Z_{4}^{1})+e(Z_{4}^{2})+e(Z_{4}^{3})-e(Z_{4}^{4})+e(Z_{4}^{5})+e(Z_{4}^{6})-e(Z_{4}^{7})-e(Z_{4}^{8})\\
& = & (q^{3}-q)(q^3-2q^2-2)
\end{eqnarray*}

We also get that:
$$
e(\overline{Z}_{4,\lambda})=(q-1)(q^2+q+1)=q^3-1
$$
$$
e(\overline{Z}_{4})=q^4-3q^3-q+3
$$
Finally,
$$
e(Z)  =  e(Z_{0})+e(Z_{1})+e(Z_{2})+e(Z_{3})+e(Z_{4}) = e(\pgl^{2})
$$
as it was expected.
Adding up all Hodge monodromy representations, we have also obtained
\begin{prop}
The monodromy of the fibration given by $\pi:\overline{Z}_{4}\rightarrow \CC -\lbrace 0, \pm 1 \rbrace$ is trivial. Therefore
$$
R(\overline{Z}_{4})=(q^{3}-1)T
$$
\end{prop}
\begin{rem}
There is a map between Hodge monodromy representations $R(\overline{X}_{4})\rightarrow R(\overline{Z}_{4})$, induced by the quotient map between fibrations $\overline{X}_{4}\rightarrow \overline{Z}_{4}$. Since $R(\overline{X}_{4})=(q^{3}-1)T+(-3q^{2}+3q)N$, we see that the Hodge monodromy representation of $\overline{Z}_{4}$ corresponds to the invariant part of $R(\overline{Z}_{4})$.

\end{rem}
\subsection*{Hodge monodromy representation of $\overline{Z}_{4}/\mathbb{Z}_{2}$}
We consider as the last basic block the Hodge monodromy representation of the fibration $\overline{Z}_{4}/\mathbb{Z}_{2} \rightarrow \CC -\lbrace \pm 2 \rbrace$, where we recall that $\mathbb{Z}_{2}$ acts on $\overline{Z}_{4}$ in the following way:
$$
(A,B,\lambda)\longrightarrow (PAP^{-1},PBP^{-1},\lambda^{-1})
$$
and fibres over $B'= \CC -\lbrace 0,\pm 1 \rbrace /\ZZ_2$, where $\lambda \sim \lambda^{-1}$,  isomorphic to $\CC -\lbrace \pm 2 \rbrace$. Let us write $\Gamma=\langle \gamma_{-1},\gamma_{0},\gamma_{1} \rangle$ and $\Gamma'=\langle \nu_{2}, \nu_{-2} \rangle$ for the fundamental groups of $B$ and $B'=B/\ZZ_2$ respectively. Under the quotient map: $\lambda \longrightarrow s=\lambda + \lambda^{-1}$, we get an exact sequence:
\begin{align*}
\Gamma\longrightarrow \Gamma' & \longrightarrow \mathbb{Z}_2\longrightarrow 0 \\
\gamma_{1} & \longrightarrow  2\nu_{2} \\
\gamma_{-1}& \longrightarrow  2\nu_{-2} \\
\gamma_{0} & \longrightarrow \nu_{2}+\nu_{-2} 
\end{align*}

Since the monodromy action of $\Gamma$ was trivial on $\overline{Z}_{4}$, we get that the actions of $2\nu_{2},2\nu_{-2}$ and $\nu_{2}+\nu_{-2}$ are trivial too. Hence the monodromy representation descends to a representation of $\mathbb{Z}_{2}$ and therefore
$$
R(\overline{Z}_{4}/\mathbb{Z}_{2})\in R(\mathbb{Z}_{2})[q]
$$
We will denote by $T,N$ the trivial and non-trivial representations of $\mathbb{Z}_{2}$ respectively. We will use again the stratification defined in previous sections that is compatible with the monodromy action. 

We need to check the monodromy when we describe a small loop around $2$ in the $s$-plane $\mathbb{C} - \lbrace \pm 2 \rbrace$, i.e, around a small loop $\nu:=\nu_{2}$ which we take as a generator of $\mathbb{Z}_{2}\cong \Gamma'/\Gamma$. To understand the action, we lift the loop to a path in the $\lambda$-plane $B$ and we fix the fibre over a point $b_0=1+\epsilon$.
The resulting lifted path goes from $1+\epsilon$ to $\frac{1}{1+\epsilon}$, followed by the $\mathbb{Z}_{2}$-action to identify the fibre.

We denote by $F^{i}=\overline{Z^{i}_{4,\lambda}}\cong S^{i}_{\lambda}\times \mathbb{C}^{\ast}$ the fibre of the fibration $\overline{Z^{i}_{4,\lambda}}\longrightarrow B'$; we will also denote by $e(F^{i})^{\nu}$ for the Hodge Deligne polynomial of $H^{\ast}_{c}(F^{i})^{\nu}$,i.e. the invariant part of the cohomology under the action of $\nu$. Proceeding for each stratum:
\begin{itemize}
\item $\overline{Z_{4}^{0}}/\mathbb{Z}_{2}, \overline{Z_{4}^{1}}/\ZZ_2$. The fibre consists of a copy of $\mathbb{C}^{\ast}$, parametrized by $Y=y^{2}$. The $\mathbb{Z}_{2}$-action takes $Y$ to $-\frac{1}Y$. Therefore
$$
e(F)=q-1, \quad e(F)^{\nu}=q
$$
If we denote $R(\overline{Z^{0}_{4}}/\mathbb{Z}_{2})=aT+bN$, we get that $a+b=e(F)$ and $a=e(F)^{\nu}$,so:
$$
R(\overline{Z^{0}_{4}}/\mathbb{Z}_{2})= R(\overline{Z^{1}_{4}})= qT-N
$$

\item $\overline{Z_{2}^{4}}/\mathbb{Z}_{2}, \overline{Z_{4}^{3}}/\ZZ_2$. We have two components, $\lbrace b\neq0, c=0 \rbrace \cup \lbrace b=0,  c\neq 0 \rbrace$, where each one is a $\mathbb{C}^{\ast}$-bundle over $\mathbb{C}^{\ast}$. The action of $\mathbb{Z}_{2}$ interchanges the components, so
$$
e(F)=2(q-1)^{2}, \quad e(F)^{\nu}=(q-1)^{2}
$$
and therefore
$$
R(\overline{Z^{2}_{4}}/\mathbb{Z}_{2})= R(\overline{Z^{3}_{4}}/\mathbb{Z}_{2})=(q-1)^{2}T+(q-1)^{2}N
$$

\item $\overline{Z_{4}^{4}}/\mathbb{Z}_{2}$. In this case, we have again two components, given by $\lbrace b=z=0 \rbrace$ and $\lbrace c=y=0 \rbrace$, where each one is isomorphic to $\mathbb{C}^{\ast}$,the stabilizer. The $\mathbb{Z}_{2}$-action interchanges them, so
$$
e(F)=2(q-1), \quad e(F)^{\nu}=(q-1)
$$
$$
R(\overline{Z^{4}_{4}}/\mathbb{Z}_{2})=(q-1)T+(q-1)N
$$
\item $\overline{Z_{4}^{5}}/\mathbb{Z}_{2}$ From the equations arising in the $SL(2,\mathbb{C})$-case, we obtained two different components, given by $a=\pm \frac{2\lambda}{\lambda +1}$, which were identified in the $\pgl$-case by $\phi$. Each of the components consisted of two parallel lines of equations:
$$
\Big( x- \frac{\lambda -1}{\lambda +1}y \Big)^{2}=\frac{1}{\lambda}
$$
which were identified by the $\varphi$-action. So the fibre is a $\mathbb{C}^{\ast}$-bundle over a single line $L$ with two points removed. The monodromy given by the $\nu$-action takes the line of equation:
$$
\sqrt{\lambda}x-\sqrt{\lambda}\frac{\lambda-1}{\lambda+1}y=1
$$
to the other parallel line (identified with the first one by the $\varphi$-action):
$$
-\sqrt{\lambda}x+\sqrt{\lambda}\frac{\lambda-1}{\lambda+1}y=1
$$
since $\sqrt{\lambda}x \rightarrow \sqrt{\lambda}x$ and $\sqrt{\lambda}\frac{\lambda -1}{\lambda +1}y \longrightarrow -\sqrt{\lambda}\frac{\lambda -1}{\lambda +1}y+2\sqrt{\lambda}x$. Therefore the monodromy action interchanges the lines which were already identified, but note that it also interchanges the removed points corresponding to $x=\pm \frac{1}{\sqrt{\lambda}}$. Therefore, $e(L)^{+}=q-1, e(L)^{-}=-1$, and using that $e(\mathbb{C}^{\ast})^{+}=q, e(\mathbb{C}^{\ast})^{-}=-1$
$$
e(F)=(q-2)(q-1)=q^2-3q+2, \quad e(F)^{\nu}=(q-1)q+1=q^2-q+1
$$
so:
$$
R(\overline{Z_{4}^{5}}/\mathbb{Z}_{2})=(q^2-q+1)T+(-2q+1)N
$$

\item $\overline{Z_{4}^{6}}/\mathbb{Z}_{2}$ The fibre is a $\mathbb{C}^{\ast}$-bundle over a family of conics over a line parametrized by $a^{2}$, with some excluded values. Looking at the previous sections, we get that
$$
e(F)=(q+1)(q-2)(q-1)
$$
The $\mathbb{Z}_{2}$-action that appears in the study of the monodromy of $\nu$ leaves the line fixed and turns $\mathbb{C}^{\ast}$ inside out.
We get that
$$
e(F)^{\nu}=q(q+1)(q-2)
$$
and therefore
$$
R(\overline{Z_{4}^{6}}/\mathbb{Z}_{2})=(q+1)(q-2)qT-(q+1)(q-2)N
$$

\item $\overline{Z_{4}^{7}}/\mathbb{Z}_{2}$ The fibre was a $\mathbb{C}^{\ast}$-bundle over two lines intersecting at the origin with two points removed in each one. Therefore,
$$
e(F)=(2q-5)(q-1).
$$
The action of $\nu$ identifies both lines and also turns $\mathbb{C}^{\ast}$ inside-out, so the quotient of the lines is a single punctured line $\mathbb{C}^{\ast}$ minus two points, and the quotient of $\mathbb{C}^{\ast}$ by this action is $\mathbb{C}$. Hence
$$
e(F)^{\nu}=((q-2)q-(q-3))
$$
and:
$$
R(\overline{Z_{4}^{7}}/\mathbb{Z}_{2})=(q^{2}-3q+3)T+(q^{2}-4q+2)N
$$

\item $\overline{Z_{4}^{8}}/\mathbb{Z}_{2}$ In this last stratum the fibre $F$ was a $\mathbb{C}^{\ast}$-bundle over a hyperbola of equation $B^{2}=A^{2}-4$, where $(A,B)\sim (-A,-B)$ and some excluded points: two at infinity, $(2,0)\sim (-2,0)$ and $\big( \frac{1+\lambda}{\sqrt{\lambda}},\frac{1-\lambda}{\sqrt{\lambda}} \big) \sim \big( -\frac{1+\lambda}{\sqrt{\lambda}},-\frac{1-\lambda}{\sqrt{\lambda}} \big)$ and $\big( \frac{1+\lambda}{\sqrt{\lambda}}, -\frac{1-\lambda}{\sqrt{\lambda}} \big) \sim \big( -\frac{1+\lambda}{\sqrt{\lambda}}, \frac{1-\lambda}{\sqrt{\lambda}} \big)$.

$\nu$ leaves the infinity points and $(2,0)\sim (-2,0)$ invariant, it turns $\mathbb{C}^{\ast}$ inside-out and also interchanges the two excluded points
$$
\frac{1+\lambda}{\sqrt{\lambda}} \longrightarrow \frac{1+\frac{1}{\lambda}}{(\sqrt{\lambda})^{-1}}=\frac{\lambda+1}{\sqrt{\lambda}}
$$
$$
\frac{1-\lambda}{\sqrt{\lambda}} \longrightarrow \frac{1-\frac{1}{\lambda}}{(\sqrt{\lambda})^{-1}}=-\frac{(1-\lambda)}{\sqrt{\lambda}}
$$
We get that
$$
e(F)=(q-1)(q-4), \quad e(F)^{\nu}=q(q-3)+1
$$
obtaining
$$
R(\overline{Z_{4}^{8}})=(q^{2}-3q+1)T+(-2q+3)N
$$
\end{itemize}
Adding all up, we finally obtain:
$$
R(\overline{Z}_{4}/\mathbb{Z}_{2})=q^{3}T-N
$$
We have proved the following proposition
\begin{prop}
The Hodge monodromy representation of the fibration $\overline{Z}_{4}/\ZZ_2 \longrightarrow \CC - \lbrace \pm 2 \rbrace$ factors through $\ZZ_2$. If we write $T,N$ for the trivial and non-trivial representation respectively, then
\begin{equation}
R(\overline{Z}_{4}/\mathbb{Z}_{2})=q^{3}T-N \label{eqn:HodgemonZ4Z2}
\end{equation}
\end{prop}

\begin{rem}
Note that if we denote by $R(\overline{Z}_{4}/\mathbb{Z}_{2})=cT+dN$, we get that:
$$
R(\overline{Z_{4}})=(c+d)T=(q^{3}-1)T
$$
Besides, applying (\ref{eqn:eoX4Z2})
$$
e(\overline{Z}_{4}/\mathbb{Z}_{2})  = (q-2)a-b =  q^4-2q^3+1
$$
\end{rem}

\section{Character varieties of curves of genus $g\geq 2$} \label{chapter:highgenus}
Let $g\geq 1$ be any natural number. We define the following sets:
\begin{itemize}
\item $\oZ_{0}^{g}= \lbrace (A_{1},B_{1},\ldots, A_{g},B_{g}) \in \pgl^{2g} \mid \prod_{i=1}^{g}[A_{i},B_{i}]= \Id\}$.
\item $\oZ_{1}^{g}= \lbrace (A_{1},B_{1},\ldots, A_{g},B_{g})\in \pgl^{2g} \mid \prod_{i=1}^{g}[A_{i},B_{i}]= -\Id\}$.
\item $\oZ_{2}^{g}= \lbrace (A_{1},B_{1},\ldots, A_{g},B_{g}) \in \pgl^{2g}\mid \prod_{i=1}^{g}[A_{i},B_{i}]= J_+=
 \begin{pmatrix} 1& 1\\ 0 & 1 \end{pmatrix}\}$.
\item $\oZ_{3}^{g}= \lbrace (A_{1},B_{1},\ldots, A_{g},B_{g}) \in \pgl^{2g}\mid \prod_{i=1}^{g}[A_{i},B_{i}]=  J_-=
 \begin{pmatrix} -1& 1\\ 0 & -1 \end{pmatrix}\}$.
\item $\oZ_{4,\lambda}^{g}= 
\lbrace (A_{1},B_{1},\ldots, A_{g},B_{g}) \in \pgl^{2g}\mid \prod_{i=1}^{g}[A_{i},B_{i}]= \xi_\lambda=
\begin{pmatrix} \lambda & 0 \\ 0 & \lambda^{-1} \end{pmatrix} \rbrace$, where $\lambda\in \CC-\{ 0, \pm 1\}$.
\item $\oZ_{4}^{g}= 
\lbrace (A_{1},B_{1},\ldots, A_{g},B_{g},\lambda) \in \pgl^{2g}\x (\CC-\{ 0, \pm 1\}) \mid 
\prod_{i=1}^{g}[A_{i},B_{i}]= \begin{pmatrix} \lambda & 0 \\ 0 & \lambda^{-1} \end{pmatrix} \rbrace$.
\end{itemize}

and the fibration 
 $$ 
  \oZ_4^g  \too  \CC-\{ 0, \pm 1\}
 $$
whose fibres are $\oZ_{4,\lambda}^{g}$. There is an action of $\ZZ_2$ on $\oZ_4^g$ given by
$ (A_{1},\ldots,B_{g},\lambda)\mapsto (P^{-1}_0A_1P_0,\ldots, P^{-1}_0B_gP_0,\lambda^{-1})$, with $P_0=
\left(\begin{array}{cc}    0 & 1 \\ 1 & 0  \end{array} \right)$ and the associated quotient fibration $\oZ_{4}/\ZZ_2 \rightarrow \CC - \lbrace \pm 2 \rbrace$. With this notation, the building blocks $\overline{Z}_{i}$ of Section \ref{sec:basicblocks} are just $\overline{Z}_{i}^{1}$.

We introduce the following notation. Let
 $$
 \tilde{e}_0^g=e(\oZ_0^g), \ \tilde{e}_1^g=e(\oZ_1^g),  \ \tilde{e}_2^g=e(\oZ_2^g), \ \tilde{e}_3^g=e(\oZ_3^g) \in \ZZ[q].
 $$
be the E-polynomials of the genus $g$ representation spaces, and let us also consider the Hodge monodromy representation of $\oZ_{4} \rightarrow \CC -\lbrace \pm 2 \rbrace$. We will see that the monodromy group is $\Gamma=\ZZ_2$, generated by $\nu_{2}$ around the puncture $2$ of 
$\CC-\{\pm 2\}$ so that the ring $R(\Gamma)$ is generated by the trivial and non-trivial representations $T$ and $N$ respectively, as it happened when the genus was equal to 1. Therefore, we write
 \begin{equation}\label{eqn:abcdg}
 R(\oZ_4^g/\ZZ_2)=\tilde{a}_g T+ \tilde{b}_{g}N,
 \end{equation}
where $\tilde{a}_g,\tilde{b}_g \in  \ZZ[q]$.
For  each $g\geq 1$, we have the following vector
 $$
  v^g=( \tilde{e}_0^g, \tilde{e}_1^g, \tilde{e}_2^g,\tilde{e}_3^g,\tilde{a}_g,\tilde{b}_g)
  $$
consisting of six polynomials. Section \ref{sec:basicblocks} gives us $v^{1}$, and we will use induction to compute $v^{g}$.

In the following sections we aim to prove the following result by induction on $g$.

\begin{thm} \label{thm:Polynomialseig}
For all $g\geq 1$, the Hodge monodromy representation of $\overline{Z}_{4}^{g}$ is of the form (\ref{eqn:abcdg}), and 
\begin{align*}
\tilde{e}_{0}^{g} & = (q^{3}-q)^{2g-1}+q(q^{2}-1)^{2g-1}+\frac{1}{2}(q(q-1)(q^{2}+q)^{2g-1}+(q-2)(q+1)(q^{2}-q)^{2g-1}),\\
\tilde{e}_{1}^{g} & = (q^{3}-q)^{2g-1}+q(q^{2}-1)^{2g-1}-\frac{1}{2}((q-1)(q^{2}+q)^{2g-1}+(q+1)(q^{2}-q)^{2g-1}), \\
\tilde{e}_{2}^{g} & = (q^{3}-q)^{2g-1}+\frac{1}{2}(-q(q^{2}+q)^{2g-1}+(q-2)(q^{2}-q)^{2g-1}), \\
\tilde{e}_{3}^{g} & = (q^{3}-q)^{2g-1}+\frac{1}{2}((q^{2}+q)^{2g-1}-(q^{2}-q)^{2g-1}), \\
\tilde{a}^{g} & = (q^{3}-q)^{2g-1}+\frac{1}{2}((q^{2}+q)^{2g-1}-(q^{2}-q)^{2g-1}), \\
\tilde{b}^{g} & = (q^{2}-1)^{2g-1}-\frac{1}{2}((q^{2}+q)^{2g-1}+(q^{2}-q)^{2g-1}).
\end{align*}
\end{thm}

To compute the E-polynomials of $\overline{Z}_{i}^{g}$, the idea for induction was already used for $\sldos$ in \cite{mamu2} and can be replicated again for $\pgl$, using this time the basic blocks computed in Section \ref{sec:basicblocks}.  Let $g=k+h$. We can consider a connected sum decomposition
$\Sigma_{k+h}=\Sigma_k\# \Sigma_h$. If we know $v_k,v_h$, and given that
$R( \oZ_4^k/\ZZ_2)$, $R( \oZ_4^h/\ZZ_2)$ are of the form (\ref{eqn:abcdg}), then we can obtain the E-polynomial information for $\overline{Z}_{g}$ via the following gluing identity
 \begin{equation} \label{eqn:uno}
\prod_{i=1}^{k+h}[A_{i},B_{i}]= C  \iff \prod_{i=1}^{k}[A_{i},B_{i}]= C \prod_{i=1}^{h}[B_{k+i},A_{k+i}].
 \end{equation}
which allows us to stratify $\oZ_{i}^{g}$ as follows.

\subsection*{E-polynomial of $\overline{Z}_{0}^{g}$}
Let us fix $C=\Id$ and using (\ref{eqn:uno}), we decompose $\oZ_{0}^{k+h}= \bigsqcup W_{i}$, where
\begin{itemize}
\item $W_{0}=\lbrace (A_{1},B_1,\ldots,A_{k+h},B_{k+h}) \mid \prod_{i=1}^{k}[A_{i},B_{i}]=  \prod_{i=1}^{h}[B_{k+i},A_{k+i}] =\Id \rbrace \cong \oZ_{0}^{k}\times  \oZ_{0}^h$.
\item $W_{1}=\lbrace (A_{1},B_1,\ldots,A_{k+h},B_{k+h}) \mid \prod_{i=1}^{k}[A_{i},B_{i}]=  \prod_{i=1}^{h}[B_{k+i},A_{k+i}] =-\Id \rbrace \cong \oZ_{1}^{k}\times  \oZ_{1}^h$.
\item $W_2=\lbrace (A_{1},B_1,\ldots,A_{k+h},B_{k+h}) \mid \prod_{i=1}^{k}[A_{i},B_{i}]=  \prod_{i=1}^{h}[B_{k+i},A_{k+i}] \sim J_{+} \rbrace$. There is a fibre bundle
$U\to \PGL(2,\CC)\times \oZ_{2}^{k}\times \oZ_{2}^h \to W_2$
\item $W_3=\lbrace (A_{1},B_1,\ldots,A_{k+h},B_{k+h}) \mid \prod_{i=1}^{k}[A_{i},B_{i}]=  \prod_{i=1}^{h}[B_{k+i},A_{k+i}] \sim J_{-} \rbrace$. Again, $U\to \PGL(2,\CC)\times \oZ_{3}^{k}\times \oZ_{3}^h \to W_3$.
\item $W_{4}= \lbrace (A_{1},B_1,\ldots,A_{k+h},B_{k+h}) \mid \prod_{i=1}^{k}[A_{i},B_{i}]=  \prod_{i=1}^{h}[B_{k+i},A_{k+i}] \sim
 \begin{pmatrix} \lambda & 0 \\ 0 & \lambda^{-1} \end{pmatrix}, \lambda \neq 0, \pm 1 \rbrace$.
\end{itemize}

To compute $e(W_4)$, we define  
 $$
 \oW_4= \lbrace (A_{1},B_1,\ldots,A_{k+h},B_{k+h},\lambda) \mid \prod_{i=1}^{k}[A_{i},B_{i}]=  \prod_{i=1}^{h}[B_{k+i},A_{k+i}]=
 \begin{pmatrix} \lambda & 0 \\ 0 & \lambda^{-1} \end{pmatrix}, \lambda \neq 0, \pm 1 \rbrace,
 $$ 
which produces fibrations 
\begin{equation}\label{eqn:oW4}
 \oW_{4} \rightarrow \CC -\{ 0,\pm 1\}, \quad \oW_{4}/\ZZ_2 \rightarrow \CC -\{ \pm 2 \}
\end{equation}
where we are taking the usual $\ZZ_2$-action given by permutation of the eigenvectors.
The Hodge monodromy representation of $\oW_{4}/\ZZ_2$ is
\begin{align} \label{eqn:oW4Z2}
R(\oW_{4}/\ZZ_2) =& \, R(\oZ_{4}^{k}/\ZZ_2) \otimes R(\oZ_{4}^h/\ZZ_2) \nonumber \\
 = &\,  (\tilde{a}_{k}\tilde{a}_h+\tilde{b}_{k}\tilde{b})T + (\tilde{a}_{k}\tilde{b}_h+\tilde{b}_{k}\tilde{a}_h ) N
\end{align}
which we will write as $R(\overline{Z}_{4}^{k}/\ZZ_2)=AT+BN$, where $A=\tilde{a}_{k}\tilde{a}_h+\tilde{b}_{k}\tilde{b}$ and $B=\tilde{a}_{k}\tilde{a}_h+\tilde{b}_{k}\tilde{b}_{h}$. Noting that $W_{4}\cong ((\pgl/D)\times \overline{W}_4)/\ZZ_2$ and applying (\ref{eqn:+-}),(\ref{eqn:eoX4Z2})
 \begin{align}\label{eqn:RX4-RX4Z2->eX4}
 e(W_4)&= e(\pgl/D)^{+}e(\oW_{4}/\ZZ_2) + e(\pgl/D)^{-}e(\oW_{4})^{-} \nonumber \\
 & = (q^2-q)e(\oW_4/\ZZ_2) +q \, e(\oW_{4}) \nonumber \\
 &= (q^2-q)((q-2)A-B) + q ((q-3)(A+B))\\
 &= (q^3-2q^2-q)A-2qB. \nonumber 
 \end{align}

The above computations and the stratification of $\oZ_0^{k+h}$ gives
 \begin{align} \label{eqn:eZ0} 
 e(\oZ_{0}^{k+h}) &=  e(\oZ_0^k)e(\oZ_0^h)+ e(\oZ_1^k)e(\oZ_1^h)+
 (q^2-1) e(\oZ_2^k)e(\oZ_2^h)+ (q^2-1) e(\oZ_3^k)e(\oZ_3^h)+ e(W_4),
 \end{align}
where we have used that $e(\PGL(2,\CC))=q^3-q$ and so $e(\PGL(2,\CC)/U)=q^2-1$.
If we set $k=g-1$, $h=1$, and substitute the values $A,B$ from (\ref{eqn:oW4Z2} ) into (\ref{eqn:RX4-RX4Z2->eX4}),
and the values of the building blocks $v^{1}$ from Section \ref{sec:basicblocks}, we obtain
 \begin{align}\label{eqn:eZ0.} \tag{$\alpha$}
 \tilde{e}_0^g=e(\oZ_{0}^{g}) = & \, + (q^4 + q^3 - q^2 - q)  e_0^{g-1}+(q^{3}-q)e_1^{g-1} \nonumber \\ 
& + ( q^5- 2 q^4 - q^3+ 2 q^2  +2 q^2 )e_2^{g-1} + (q^5-q^3)e_3^{g-1}  \nonumber \\ 
&+ (q^6-2q^5-q^4+2q) a_{g-1} + (-2q^4-q^3+2q^2+q) b_{g-1}. \nonumber
 \end{align}

\subsection*{E-polynomial of $\overline{Z}_{1}^{g}$}
When $C=-\Id$, using (\ref{eqn:uno}) we can stratify $\oZ_{1}^{k+h}= \bigsqcup W'_{i}$, where
\begin{itemize}
\item $W'_{0}=\lbrace (A_{1},B_1,\ldots,A_{k+h},B_{k+h}) \mid \prod_{i=1}^{k}[A_{i},B_{i}]=  -\prod_{i=1}^{h}[B_{k+i},A_{k+i}] =\Id \rbrace \cong \oZ_{0}^{k}\times  \oZ_{1}^h$.
\item $W'_{1}=\lbrace (A_{1},B_1,\ldots,A_{k+h},B_{k+h}) \mid \prod_{i=1}^{k}[A_{i},B_{i}]=  -\prod_{i=1}^{h}[B_{k+i},A_{k+i}] =-\Id \rbrace \cong \oZ_{1}^{k}\times  \oZ_{0}^h$.
\item $W'_2=\lbrace (A_{1},B_1,\ldots,A_{k+h},B_{k+h}) \mid \prod_{i=1}^{k}[A_{i},B_{i}]=  -\prod_{i=1}^{h}[B_{k+i},A_{k+i}] \sim J_{+} \rbrace$, and there is a fibre bundle
$U\to \PGL(2,\CC)\times \oZ_{2}^{k}\times \oZ_{3}^h \to W_2$.
\item $W'_3=\lbrace (A_{1},B_1,\ldots,A_{k+h},B_{k+h}) \mid \prod_{i=1}^{k}[A_{i},B_{i}]=  \prod_{i=1}^{h}[B_{k+i},A_{k+i}] \sim J_{-} \rbrace$ and $U\to \PGL(2,\CC)\times \oZ_{3}^{k}\times \oZ_{2}^h \to W_3$.
\item $W'_{4}= \lbrace (A_{1},B_1,\ldots,A_{k+h},B_{k+h}) \mid \prod_{i=1}^{k}[A_{i},B_{i}]=  -\prod_{i=1}^{h}[B_{k+i},A_{k+i}] \sim
 \begin{pmatrix} \lambda & 0 \\ 0 & \lambda^{-1} \end{pmatrix}, \lambda \neq 0, \pm 1 \rbrace$.
\end{itemize}

The Hodge monodromy representation $R(\overline{W}_{4}'/\ZZ_2)$ is given by $R(\overline{W}_{4}'/\ZZ_2)=R(\oZ_{4}^{h})\otimes R(\tau^{\ast}(\oZ_{4}^{k}))$, where $\tau(\lambda)=-\lambda$. Since the monodromy is an involution, $R(\tau^{\ast}(\oZ_{4}^{k}))=R(\oZ_{4}^{k})$, so we obtain that $R(\overline{W}_{4}'/\ZZ_2)=R(\overline{W}_{4}/\ZZ_2)$ and therefore $e(W_{4}')=e(W_{4})$.

We obtain
 \begin{align} \label{eqn:eZ1} 
 e(\oZ_{1}^{k+h}) &=  e(\oZ_0^k)e(\oZ_1^h)+ e(\oZ_1^k)e(\oZ_1^h)+
 (q^2-1) e(\oZ_2^k)e(\oZ_3^h)+ (q^2-1) e(\oZ_3^k)e(\oZ_2^h)+ e(W_4),
 \end{align}
For $k=g-1$, $h=1$, replacing the values of the building blocks of $v^{1}$,  we obtain
 \begin{align}\label{eqn:eZ1.} \tag{$\beta$}
 \tilde{e}_1^g=e(\oZ_{1}^{g}) = & \, (q^{3}-q) e_0^{g-1}+ (q^4 + q^3 - q^2 - q) e_1^{g-1} \nonumber \\ 
& + (q^5-q^3)e_2^{g-1} + ( q^5- 2 q^4 - q^3+ 2 q^2  +2 q^2 )e_3^{g-1}  \nonumber \\ 
&+ (q^6-2q^5-q^4+2q) a_{g-1} + (-2q^4-q^3+2q^2+q) b_{g-1}. \nonumber
 \end{align}
 
 \subsection*{E-polynomials of $\oZ_{2}^{g}, \oZ_{3}^{g}$}
 \begin{prop} \label{prop:z2z3g-1}
 The E-polynomials of the Jordan cases, $\tilde{e}_{2}^{g},\tilde{e}_{3}^{g}$, can be expressed in terms of the E-polynomials of $L^{g-1}=(\tilde{e}_{0}^{g-1},\tilde{e}_{1}^{g-1},\tilde{e}_{2}^{g-1},\tilde{e}_{3}^{g-1},\tilde{a}^{g-1},\tilde{b}^{g-1})$ as
 \begin{align*}
 \tilde{e}_{2}^{g} & = (q^{3}-2q^{2})\tilde{e}_{0}^{g-1}+q^{3}\tilde{e}_{1}^{g-1}+(q^{5}+q^{4}-3q^{3}+3q^{2})\tilde{e}_{2}^{g-1}+(q^{5}-3q^{3})\tilde{e}_{3}^{g-1} \\
 & +(q^6-2q^5-3q^4+4q^3)\tilde{a}^{g-1} +(-2q^{4}+2q^{3}) \tilde{b}^{g-1} \\
  \tilde{e}_{3}^{g} & = q^{3} \tilde{e}_{0}^{g-1}+(q^{3}-2q^{2})\tilde{e}_{1}^{g-1}+(q^{5}-3q^{3})\tilde{e}_{2}^{g-1}+(q^{5}+q^{4}-3q^{3}+3q^{2})\tilde{e}_{3}^{g-1} \\
  & +(q^6-2q^5-3q^4+4q^3)\tilde{a}^{g-1} +(-2q^{4}+2q^{3}) \tilde{b}^{g-1}
  \end{align*}
 \end{prop}
 \begin{proof}
For the positive Jordan case, where $C=J_{+}$, relation (\ref{eqn:uno}) can be written as
$$
\nu =\prod_{i=1}^{h}[B_{k+i},A_{k+i}] =\begin{pmatrix} a & b \\ c & d \end{pmatrix}, \qquad 
\delta = \prod_{i=1}^{k}[A_{i},B_{i}] = J_{+}\nu = \begin{pmatrix} a+c & b+d \\ c & d \end{pmatrix}.
$$
Let $t_{1}=\tr \nu$, $t_{2}=\tr \delta$. Note that $c=t_{2}-t_{1}$. A stratification of $\oZ_{2}^{g}$ can be given using the map $\oZ^{g}\longrightarrow \CC^{2}$ given by $\rho \mapsto (\tr \nu, \tr \delta)$, as it was done in \cite{lomune} for $\sldos$-character varieties. The stratification distinguishes when any of the traces is equal to $\pm 2$, the case $t_{1}=t_{2}$ and the generic case, when $t_{1}\neq t_{2}\neq \pm 2$. The matrices $\nu$ and $\delta$ are defined in terms of commutators, so they are invariant for the torus action when passing from $\sldos$ to $\pgl$ and therefore, the $\pgl$-case can be obtained from the stratification given for the $\sldos$-case just by quotienting each stratum by the torus action. Each of those quotients corresponds to the spaces $\oZ_{i}^{k}$. In other words, the stratification given in \cite{mamu2} for the Jordan positive case provides a stratification for $\overline{Z}_{2}^{g}$ if we replace the basic blocks $\oX_{i}^{k}$ by their quotients, the $\pgl$-representation spaces $\oZ_{i}^{k}$. Looking at \cite[Section 6]{mamu2}, we obtain
\begin{align}
e(\oZ_{2}^{k+h}) = & \, e(\oZ_{2}^k)e(\oZ_{0}^{h})+e(\oZ_{0}^k)e(\oZ_{2}^{h})-2e(\oZ^k_{2})e(\oZ_{2}^{h})
+e(\oZ_{3}^k)e(\oZ_{1}^{h})+e(\oZ_{1}^k)e(\oZ_{3}^{h})-2e(\oZ_{3}^k)e(\oZ_{3}^{h}) \nonumber \\
& +q\, (e(\oZ^k_{2})+e(\oZ^k_{3})+e(\oZ_{4}^k/\ZZ_{2}))( e(\oZ_{2}^{h})+ e(\oZ_{3}^{h})+e(\oZ_{4}^{h}/\ZZ_{2})) \label{eqn:Z2k+h}
+q\,( e(\overline{V}_{5})- e(\overline{V}_{5}/\ZZ_2)).
\end{align}
where $\overline{V}_{5} \longrightarrow \CC -\lbrace 0, \pm 1 \rbrace$ is the product bundle $\overline{V_{5}}\cong \overline{Z}_{4}^{k}\times_{\CC-\lbrace 0,\pm 1 \rbrace} \overline{Z}^{h}_{4}$, which has Hodge monodromy representation $R(\overline{V}_{5})=R(\overline{Z}_{4}^{h})\otimes R(\overline{Z}_{4}^{k})=(\tilde{a}_{k}\tilde{a}_{h}+\tilde{b}_{k}\tilde{b}_{h})T+(\tilde{a}_{k}\tilde{b}_{h}+\tilde{a}_{h}\tilde{b}_{k})N$. Setting $k=g-1,h=1$ in (\ref{eqn:Z2k+h}), and replacing the values for $v^{1}$, we get
\begin{align*} \tag{$\gamma$}
\tilde{e}_{2}^{g} & =(q^3-2q^{2})\tilde{e}_{0}^{g-1}+q^{3}\tilde{e}_{1}^{g-1}+(q^{5}+q^{4}-3q^{3}+3q^{2})\tilde{e}_{2}^{g-1} + (q^{5}-3q^{3})\tilde{e}_{3}^{g-1} \\
& + (q^{6}-2q^{5}-3q^{4}+4q^{3})\tilde{a}^{g-1}+(-2q^{4}+2q^{3})\tilde{b}^{g-1}
\end{align*}
The negative Jordan case is worked out similarly. Relation (\ref{eqn:uno}), and the stratification used in \cite[Section 7]{mamu2} lead to the following relation
\begin{align*}
e(\oZ_{3}^{k+h}) = & \,e(\oZ_{2}^k)e(\oZ_{1}^{h})+e(\oZ_{0}^k)e(\oZ_{3}^{h})-2e(\oZ^k_{2})e(\oZ_{2}^{h})
+e(\oZ_{3}^k)e(\oZ_{0}^{h})+e(\oZ_{1}^k)e(\oZ_{2}^{h})-2e(\oZ_{2}^k)e(\oZ_{3}^{h}) \\
& +q\, (e(\oZ^k_{2})+e(\oZ^k_{3})+e(\oZ_{4}^k/\ZZ_{2}))( e(\oZ_{2}^{h})+ e(\oZ_{3}^{h})+e(\oZ_{4}^{h}/\ZZ_{2}))
+q\,( e(\overline{V}{}'_{5})- e(\overline{V}{}'_{5}/\ZZ_2)).
\end{align*}
where $\overline{V}'_{5}\cong \overline{Z}_{4}\times \tau^{\ast}(\overline{Z}_{4})$. Setting $k=g-1,h=1$ and using the values of the building blocks in $v^{1}$ gives
\begin{align*} \tag{$\delta$}
\tilde{e}_{3}^{g} & = q^{3} \tilde{e}_{0}^{g-1}+(q^{3}-2q^{2})\tilde{e}_{1}^{g-1}+(q^{5}-3q^{3})\tilde{e}_{2}^{g-1}+(q^{5}+q^{4}-3q^{3}+3q^{2})\tilde{e}_{3}^{g-1} \\
  & +(q^6-2q^5-3q^4+4q^3)\tilde{a}^{g-1} +(-2q^{4}+2q^{3}) \tilde{b}^{g-1}
\end{align*}
 \end{proof}
 
 \subsection*{E-polynomials of $R(\overline{Z}_{4}^{g})$}
 \begin{prop}
 The Hodge monodromy representation $R(\oZ_{4}^{g})$ can be written in terms of the elements of genus $1$ and $g-1$, $v^{1},L^{g-1}$, as
 \begin{align*}
 R(\overline{Z}_{4}^{g})=(\tilde{a}^{g}+\tilde{b}^{g})T & = ((q^{3}-1)\tilde{e}_{0}^{g-1}+(q^{3}-1)\tilde{e}_{1}^{g-1} \\
 & + (q^5-3q^3+2q^2)\tilde{e}_{2}^{g-1}+(q^5-3q^3+2q^2)\tilde{e}_{3}^{g-1} \\
 & +(q^6-2q^5-2q^4+4q^3-3q^2+2)a^{g-1} + (-q^4-2q^2+2q+1)b^{g-1})T
 \end{align*}
 \end{prop}
 \begin{proof}
The idea is similar to the proof of Proposition \ref{prop:z2z3g-1}. Relation (\ref{eqn:uno}) is given in this case by the equations
$$
\nu =\prod_{i=1}^{h}[B_{k+i},A_{k+i}] =\begin{pmatrix} a & b \\ c & d \end{pmatrix}, \qquad 
\delta = \prod_{i=1}^{k}[A_{i},B_{i}] = \xi_\lambda \nu =\begin{pmatrix} \lambda a & \lambda b \\ \lambda^{-1} c & \lambda^{-1}d \end{pmatrix}.
$$
where $t_{1}=\tr \nu$, $t_{2}= \tr \delta$ and $\lambda$ determine $a,d$, and $bc=ad-1$. The monodromy of $\overline{Z}_{4}\mapsto \CC - \lbrace 0, \pm 1 \rbrace$ can be studied stratifying $\oZ_{4}=\sqcup_{i=1}^{7} W^{''}_{i}$ in terms of the different values for $(t_{1},t_{2})\in \CC^{2}$, where each $W^{''}_{i}$ can again be obtained as the quotient by the torus action of its corresponding stratum in the $\sldos$-stratification obtained in \cite{mamu2}. Therefore, the substitution of each $\sldos$-block $\oX^{k}_{i}$ by its corresponding quotient $\oZ^{k}_{i}$ in \cite[Section 8]{mamu2} leads to the relation
\begin{align}
R(\oZ_{4}^{k+h}) & =(q-1)(e(\oZ_{2}^{k})+e(\oZ_{3}^{k})+e(\oZ_{4}^{k}/\ZZ_2))(e(\oZ_{2}^{h})+e(\oZ_{3}^{h})+e(\oZ_{4}^{h}/\ZZ_2))T \nonumber \\
& +(e(\oZ_{0}^{k})+e(\oZ_{1}^{k})+(q-1)e(\oZ_{2}^{k})+(q-1)e(\oZ_{3}^{k}))R(\oZ_{4}^{h}) \label{Z4k+h}\\
& +(e(\oZ_{0}^{h})+e(\oZ_{1}^{h})+(q-1)e(\oZ_{2}^{h})+(q-1)e(\oZ_{3}^{h}))R(\oZ_{4}^{k}) + qR(\overline{V}'') \nonumber
\end{align}
where $\overline{V}''$ is the quotient of the fibration $R(\overline{Z}_{6})$ that appears in \cite[Section 8]{mamu2}. Its Hodge monodromy representation is $R(\overline{V}'')=(q-5)(a_{k}+b_{k})(a_{h}+b_{h})T$. Setting $k=g-1,h=1$ in (\ref{Z4k+h}) gives the desired result.
 \end{proof}
Finally, an additional equation for $v^{g}$ can be obtained from the fact that $\sldos^{2g}=\sqcup_{i}\oZ_{i}$. If we recall that $e(\sldos)=(q^{3}-q)$ and we apply (\ref{eqn:eoX4Z2}), we obtain
$$
(q^{3}-q)^{2g}=\tilde{e}_{0}^{g}+\tilde{e}_{1}^{g}+(q^{2}-1)\tilde{e}_{2}^{g}+(q^{2}-1)\tilde{e}_{3}^{g}+(q-2)\tilde{a}^{g}-\tilde{b}^{g}
$$
but also, for $g-1$,
$$
(q^{3}-q)^{2g}=(q^{3}-q)^{2}(q^{3}-q)^{2g-2}=(q^{3}-q)^{2}(\tilde{e}_{0}^{g-1}+\tilde{e}_{1}^{g-1}+(q^{2}-1)\tilde{e}_{2}^{g-1}+(q^{2}-1)\tilde{e}_{3}^{g-1}+(q-2)\tilde{a}^{g-1}-\tilde{b}^{g-1})
$$
Combining both equations produces the following relation
\begin{align}
\tilde{e}_{0}^{g}+\tilde{e}_{1}^{g}+(q^{2}-1)\tilde{e}_{2}^{g}+(q^{2}-1)\tilde{e}_{3}^{g}+(q-2)\tilde{a}^{g}-\tilde{b}^{g} & =(q^{3}-q)^{2}(\tilde{e}_{0}^{g-1}+\tilde{e}_{1}^{g-1} \nonumber \\
& +(q^{2}-1) \tilde{e}_{2}^{g-1}+(q^{2}-1)\tilde{e}_{3}^{g-1}+(q-2)\tilde{a}^{g-1}-\tilde{b}^{g-1}) \label{eqn:pglgenusrelation}
\end{align}
If we substitute in (\ref{eqn:pglgenusrelation}) the values of $\tilde{e}_{0}^{g},\tilde{e}_{1}^{g}, \tilde{e}_{2}^{g}, \tilde{e}_{3}^{g}$ obtained in ($\alpha$), ($\beta$), ($\gamma$),($\delta$) we obtain
\begin{equation}
\tilde{a}^{g} = q^{3}\tilde{e}_{0}^{g-1} + q^{3}\tilde{e}_{1}^{g-1}+(q^{5}-3q^{3})\tilde{e}_{2}^{g-1}+(q^{5}-3q^{3})\tilde{e}_{3}^{g-1}+(q^{6}-2q^{5}-2q^{4}+4q^{3}+q^{2})\tilde{a}^{g-1}-2q^{4}\tilde{b}^{g-1} \tag{$\epsilon$}
\end{equation}
\begin{equation}
\tilde{b}^{g}= -\tilde{e}_{0}^{g-1} -\tilde{e}_{1}^{g-1}+2q^{2}\tilde{e}_{2}^{g-1}+2q^{2}\tilde{e}_{3}^{g-1}+(-4q^{2}+2)\tilde{a}^{g-1}+(q^{4}-2q^{2}+2q+1)\tilde{b}^{g-1} \tag{$\zeta$}
\end{equation}
Equations $(\alpha),(\beta),(\gamma),(\delta),(\epsilon),(\zeta)$ can be summarized in the following proposition:
\begin{prop} \label{prop:matrixrecurrence}
There exists a $6\times 6$ matrix $M$ with coefficients in $\ZZ[q]$ such that the vectors $v^{g}=(\tilde{e}_{0}^{g},\tilde{e}_{1}^{g},\tilde{e}_{2}^{g},\tilde{e}_{3}^{g},\tilde{a}^{g},\tilde{b}^{g})$ satisfy, for all $g\geq 1$.
\begin{equation} \label{eqn:matrixrecurrence}
(v^{g})^{t}=M(v^{g-1})^{t}
\end{equation}
The starting vector is $v^{0}=(1,0,0,0,0,0)$ and $M$ is the following matrix
\begin{equation} \label{matrixg-1-->g}
\scalemath{0.65}{\left( \begin{array}{c@{\hspace{2em}}c@{\hspace{2em}}c@{\hspace{2em}}c@{\hspace{2em}}c@{\hspace{2em}}c@{\hspace{2em}}c@{\hspace{2em}}c}
 q^{4}+q^{3}-q^{2}-q & q^3-q & q^5 -2q^4-q^3+2q^{2} & q^{5}-q^{3}   & q^6 -2q^5 -q^4+2q & -2q^{4}-q^{3}+2q^{2}+q \\
 & & & & & \\
 q^{3}-q  & q^{4}+q^{3}-q^{2}-q     & q^{5}-q^{3}       & q^{5}-2q^{4}-q^{3}+2q^{2}      & q^{6}-2q^{5}-q^{4}+2q  & -2q^{4}-q^{3}+2q^{2}+q  \\
& & & & & \\ 
q^{3}-2q^{2} & q^{3} & q^{5}+q^{4}-3q^{3}+3q^{2}  & q^{5}-3q^{3} & q^{6}-2q^{5}-3q^{4}+4q^{3} & -2q^{4}+2q^{3} \\
 & & & & &  \\
q^3 & q^3-2q^2 & q^5-3q^3  & q^5+q^{4}-3q^{3}+3q^{2} & q^{6}-2q^{5}-3q^{4}+4q^{3} & -2q^{4}+2q^{3} \\
 & & & & &  \\
q^3& q^3  & q^{5}-3q^{3} & q^{5}-3q^{3} & q^{6}-2q^{5}-2q^{4}+4q^{3}+q^{2}   & -2q^{4}   \\
 & & & & & \\
-1 & -1 & 2q^{2} & 2q^{2} & -4q^{2}+2 & q^{4}-2q^{2}+2q+1 \\
& & & & &  \\

\end{array} \right)}
\end{equation}
\end{prop}
As a consequence, we obtain closed formulas for the E-polynomials $\tilde{e}_{i}^{g}$ of the representation spaces $\overline{Z}_{i}^{g}, i=0 \ldots 3$ and for the Hodge monodromy representation of $\overline{Z}_{4}^{g}$ for arbitrary genus.
\begin{thm}\label{thm:epolysofei}
For all $g\geq 1$,
\begin{align*}
\tilde{e}_{0}^{g} & = (q^{3}-q)^{2g-1}+q(q^{2}-1)^{2g-1}+\frac{1}{2}(q(q-1)(q^{2}+q)^{2g-1}+(q-2)(q+1)(q^{2}-q)^{2g-1})\\
\tilde{e}_{1}^{g} & = (q^{3}-q)^{2g-1}+q(q^{2}-1)^{2g-1}-\frac{1}{2}((q-1)(q^{2}+q)^{2g-1}+(q+1)(q^{2}-q)^{2g-1}) \\
\tilde{e}_{2}^{g} & = (q^{3}-q)^{2g-1}+\frac{1}{2}(-q(q^{2}+q)^{2g-1}+(q-2)(q^{2}-q)^{2g-1}) \\
\tilde{e}_{3}^{g} & = (q^{3}-q)^{2g-1}+\frac{1}{2}((q^{2}+q)^{2g-1}-(q^{2}-q)^{2g-1}) \\
\tilde{a}^{g} & = (q^{3}-q)^{2g-1}+\frac{1}{2}((q^{2}+q)^{2g-1}-(q^{2}-q)^{2g-1}) \\
\tilde{b}^{g} & = (q^{2}-1)^{2g-1}-\frac{1}{2}((q^{2}+q)^{2g-1}+(q^{2}-q)^{2g-1})
\end{align*}
\end{thm}
\begin{proof}
It is a direct consequence of Proposition \ref{prop:matrixrecurrence}. From (\ref{eqn:matrixrecurrence}), we can diagonalize $M=PDP^{-1}$ over $\ZZ[q]$, where $P$ and $D$ are
$$
\scalemath{0.76}{
P=\begin{pmatrix}
-(q^{2}-2q-3) & -(q+1) & -(q-1)^{2} & q-1 & q & 1 \\ 0 & q+1 & 0 & -(q-1) & q & 1 \\ -(q-3) & -1 & q-1 & -1 & 0 & 1 \\ 0 & 1 & 0 & 1 & 0 & 1 \\ 1 & 0 & -1 & 0 & 0 & 1 \\ 1 & 0 & 1 & 0 &  1 & 0
\end{pmatrix}, \: D = \begin{pmatrix}
(q^{2}-q)^{2} & 0 & 0 & 0 & 0 & 0 \\ 0 & (q^{2}-q)^{2} & 0 & 0 & 0 & 0 \\ 0 & 0 & (q^{2}+q)^{2} & 0 & 0 & 0 \\ 0 & 0 & 0 & (q^{2}+q)^{2} & 0 & 0 \\ 0 & 0 & 0 & 0 & (q^{2}-1)^{2} & 0 \\ 0 & 0 & 0 & 0 & 0 & (q^{3}-q)^{2}
\end{pmatrix}}
$$
Setting $v^{g}=PD^{g}P^{-1}$ gives the desired formulas.
\end{proof}
From this result we derive the E-polynomials of the moduli spaces of Theorem \ref{thm:epolysmodulis} where the GIT quotient is geometric, dividing by the respective stabilizer.
\begin{thm} For all $g\geq 1$,
\begin{align*}
e(\mathcal{M}_{-\Id}^{g}) & = e(\oZ_{1}^{g})/(q^{3}-q) = (q^{3}-q)^{2g-2}+(q^{2}-1)^{2g-2}-\frac{1}{2}((q^{2}+q)^{2g-2}+(q^{2}-q)^{2g-2}) \\
e(\mathcal{M}_{J_{+}}^{g}) & = e(\oZ_{2}^{g})/q = (q^{2}-1)(q^{3}-q)^{2g-2}+\frac{1}{2}(-q(q+1)(q^{2}+q)^{2g-2}+(q-2)(q-1)(q^{2}-q)^{2g-2}) \\
e(\mathcal{M}_{J_{-}}^{g}) & =  e(\oZ_{3}^{g})/q = (q^{2}-1)(q^{3}-q)^{2g-2}+\frac{1}{2}((q+1)(q^{2}+q)^{2g-2}-(q-1)(q^{2}-q)^{2g-2}) \\
e(\mathcal{M}_{\lambda}^{g}) & = e(\oZ_{4,\lambda}^{g})/(q-1)= (\tilde{a}^{g}+\tilde{b}^{g})/(q-1) = (q^{2}+q)(q^{3}-q)^{2g-2}+(q+1)(q^{2}-1)^{2g-2} -q(q^{2}-q)^{2g-2}
\end{align*}
\end{thm}
\begin{cor}
Let $g\geq 1$. The Hodge monodromy representation of the space $\mathcal{M}_{par}$ that parametrizes the set of parabolic $\pgl$-character varieties for $\lambda\neq 0,\pm 1$ given by
$$
R(\mathcal{M}_{\lambda})=((q^{2}+q)(q^{3}-q)^{2g-2}+(q+1)(q^{2}-1)^{2g-2}-q(q^{2}-q)^{2g-2})T
$$
\end{cor}
\begin{proof}
We have that $\mathcal{M}_{par}=\bigsqcup_{\lambda \neq 0,\pm 1} \mathcal{M}_{\lambda} = \oZ_{4}^{g}/\CC^{\ast}$. Since $R(\oZ_{4}^{g})=(\tilde{a}^{g}+\tilde{b}^{g})T$, the result follows from Theorem \ref{thm:epolysofei} dividing by $e(\CC^{\ast})=q-1$.
\end{proof}
We obtain the following corollary, which was also true for $\sldos$:
 \begin{cor} \label{cor:Hauselrel}
 For all $g\geq 1$,
 $$
 e(\mathcal{M}^{g}_{J_{-}}(\pgl))+(q+1)e(\mathcal{M}^{g}_{-\Id}(\pgl))=e(\mathcal{M}^{g}_{\lambda}(\pgl))
 $$
 \end{cor}

\section{E- polynomial of the moduli space of representations}
To compute the E-polynomial of $\mathcal{M}_{\Id}^{g}$, we will first compute the E-polynomial of the reducible locus $Z_{0}^{red}$ using stratifications. Then
$$
e(\mathcal{M}_{\Id})=e(\mathcal{M}_{\Id}^{irr})+e(\mathcal{M}_{\Id}^{red}),
$$
where the first E-polynomial can be computed as $e(\mathcal{M}_{\Id}^{irr})=e(Z_{0})-e(Z_{0}^{red})/e(\pgl)$ (the quotient is geometric) and the latter can be computed directly as follows.

Recall that a reducible representation given by $(A_{1},B_{1},\ldots,A_{g},B_{g})\in \sldos^{2g}$ is S-equivalent to:
\begin{equation} \label{eqn:redrepres}
\left( \begin{pmatrix} \lambda_{1} & 0 \\ 0 & \lambda_{1}^{-1} \end{pmatrix}, \begin{pmatrix} \lambda_{2} & 0 \\ 0 & \lambda_{2}^{-1} \end{pmatrix}, \ldots, \begin{pmatrix} \lambda_{2g} & 0 \\ 0 & \lambda_{2g}^{-1} \end{pmatrix} \right)
\end{equation} 
under the $\ZZ_{2}$-action $(\lambda_{1},\ldots,\lambda_{2g})\sim (\lambda_{1}^{-1},\ldots,\lambda_{2g}^{-1})$, and the $\ZZ_{2}^{2g}$-action given by $\lambda_{i}\sim -\lambda_{i}, i=1\ldots 2g$. To get rid of the latter, we can take the parameters $\mu_{i}=\lambda_{i}^{2} \in \CC^{\ast}$; the first action takes then $(\mu_{1},\ldots, \mu_{2g})$ to $(\mu_{1}^{-1},\ldots, \mu_{2g}^{-1})$. We obtain:
\begin{align*}
e(\mathcal{M}_{\Id}^{red}) & = e((\CC^*)^{2g}/\ZZ_2) \\
& = (e(\CC^*)^{+})^{2g}+ \binom{2g}{2}(e(\CC^*)^{+})^{2g-2}(e(\CC^*)^{-})^{2}+\ldots + \binom{2g}{2g-2}(e(\CC^*)^{+})^{2}(e(\CC^*)^{-})^{2g-2}+ (e(\CC^*)^{-})^{2g} \\
& = \frac{1}{2} \left( (q-1)^{2g} +(q+1)^{2g} \right).
\end{align*}
A reducible representation will happen if there is a common eigenvector. If we lift the representation to $\sldos$, with respect to a suitable basis it will be of the form:
$$
\rho=\left( \begin{pmatrix} \lambda_{1} & a_{1} \\ 0 & \lambda_{1}^{-1} \end{pmatrix}, \begin{pmatrix} \lambda_{2} & a_{2} \\ 0 & \lambda_{2}^{-1} \end{pmatrix}, 
\ldots, \begin{pmatrix} \lambda_{2g} & a_{2g} \\ 0 & \lambda_{2g}^{-1} \end{pmatrix} \right),
$$
that can be identified with $(\CC^{\ast}\times \CC)^{2g}$. The condition $\prod_{i=1}^{g}[A_{i},B_{i}]=\Id$ is rewritten as
\begin{equation} \label{eqn:redlocusrelation}
\sum_{i=1}^{g} \lambda_{2i}(\lambda_{2i-1}^{2}-1)a_{2i} - \lambda_{2i-1}(\lambda_{2i}^{2}-1)a_{2i-1} = 0.
\end{equation}
For the $\pgl$-case, we need to consider the quotients for the $2g$ $\ZZ_{2}$-actions given by $(a_{i},\lambda_{i})\sim (-a_{i},-\lambda_{i})$, $i=1\ldots 2g$. As it was done in the $\sldos$-case, we can stratify the reducible locus in four cases, which are in fact compatible with the $\ZZ_{2}$-actions, and compute in this way the $\pgl$-case. They are
\begin{itemize}
\item $R_{1}$, given by $(a_{1},\ldots, a_{2g}) \in \CC \langle \lambda_{1}-\lambda_{1}^{-1},\ldots, \lambda_{2g}-\lambda_{2g}^{-1} \rangle$, and $(\lambda_{1},\ldots, \lambda_{2g})\neq (\pm 1, \ldots, \pm 1)$. We can conjugate the representation to the diagonal form (\ref{eqn:redrepres}) and assume that $a_{i}=0$, so that the stabilizer is the set of diagonal matrices $D$. To take into account the $\ZZ_{2}^{2g}$-action, we may define $\mu_{i}=\lambda_{i}^{2}$, so that we obtain a slice for the conjugacy action parametrized by $(\mu_{1},\ldots,\mu_{2g})\in B:=(\CC^{\ast})^{2g}- \lbrace (1,\ldots, 1) \rbrace$.
We have a diagram as in \cite{mamu2},
$$
 \begin{array}{ccccc} 
  \CC^* &\to & \PGL(2,\CC) \x B &\to & \tilde R_1 \nonumber\\
  || & & \downarrow & & \downarrow \\  
  \CC^* &\to & (\PGL(2,\CC) \x B)/\ZZ_2 &\to & R_1 \nonumber
  \end{array}
  $$
Now, since
$$
e(B/\ZZ_2)=\frac{1}{2}((q-1)^{2g}+(q+1)^{2g})-1 
$$
the total E-polynomial of this stratum is
\begin{align*}
e(R_{1}) & = (q^2-q) e(B/\ZZ_2)+ q\, e(B) \\
& = (q^2-q) \left( \frac{1}{2}((q-1)^{2g}+(q+1)^{2g})-1\right) +q((q-1)^{2g}-1)\\
& = (q^3-q)\frac{1}{2}\left( (q-1)^{2g-1}+(q+1)^{2g-1} \right)-q^{2}.
\end{align*}

\item $R_{2}$, defined by  $(a_{1},\ldots, a_{2g}) \not \in \CC \langle \lambda_{1}-\lambda_{1}^{-1},\ldots, \lambda_{2g}-\lambda_{2g}^{-1} \rangle$, and $(\lambda_{1},\ldots, \lambda_{2g})\neq (\pm 1, \ldots, \pm 1)$. Condition (\ref{eqn:redlocusrelation}) defines a hyperplane $H$ in $\CC^{2g}$, from which we exclude a line corresponding to the condition for $(a_{1},\ldots, a_{n})$. The $\ZZ_2^{2g}$-actions take $(a_{i},\lambda_{i})$ to $(-a_{i},-\lambda_{i})$. Writing $A=(\CC^{\ast})^{2g}-\lbrace (\pm 1, \ldots, \pm 1)\rbrace$, we get a fibration $H-l \rightarrow \mathcal{A} \rightarrow A$, so that
\begin{align*}
e(\mathcal{A}/\ZZ_2^{2g}) & = e(A/\ZZ_2^{2g})e(H-l/\ZZ_2^{2g}) \\
 & = ((q-1)^{2g}-1)(q^{2g-1}-q).
\end{align*}
since the quotient of the hyperplane $H\cong \CC^{2g-1}$ under each of the $\ZZ_2$-actions is again isomorphic to $\CC^{2g-1}$ (and therefore, for each action, $e(H-l)^{-}=0$). Now, writing $U'\cong D\times U$ for the upper triangular matrices, the set of reducible representations can be given as a fibration $U'\rightarrow \mathcal{A}/\ZZ_2^{2g} \times \pgl \rightarrow R_{2}$ (note that the $\ZZ_2$-actions do not interfere with the conjugacy action of $\pgl$). Therefore
\begin{align*}
e(R_{2}) & = ((q-1)^{2g}-1)(q^{2g-1}-q)(q^{3}-q)/(q^{2}-q) \\
& = (q+1)(q^{2g-1}-q)((q-1)^{2g}-1).
\end{align*}

\item $R_{3}$, given by $(a_{1},\ldots, a_{2g})=(0,\ldots,0)$, $(\lambda_{1},\ldots, \lambda_{2g})=(\pm 1, \ldots, \pm 1)$, which is precisely when $A_{i},B_{i}= \pm \Id$. There are $2^{2g}$ points, which get identified under the $\ZZ_2^{2g}$-action. So
$$
e(R_{3}) = 1.
$$

\item $R_{4}$, which is given by the case $(a_{1},\ldots, a_{2g}) \neq (0,\ldots, 0)$, $(\lambda_{1},\ldots, \lambda_{2g})=(\pm 1, \ldots, \pm 1)$. There is at least one matrix of Jordan type, so $D$ acts projectivizing the set $(a_{1},\ldots, a_{2g}) \in \CC^{2g}-\lbrace (0,\ldots, 0) \rbrace$ and the $\ZZ_2^{2g}$-actions allow us to reduce to the case where $\lambda_{1}=\ldots=\lambda_{2g}=1$. The stabilizer is isomorphic to $U$. Therefore
$$
e(R_{4})=(q^{2g}-1)(q+1)
$$
\end{itemize}
We have obtained that the E-polynomial of the reducible locus in $X_{0}^{g}$ is
\begin{align*}
e(R) & = e(R_{1})+e(R_{2})+e(R_{3})+e(R_{4}) \\
& = (q^{3}-q) \left( \frac{1}{2}((q+1)^{2g-1} -(q-1)^{2g-1}) +q^{2g-2}+(q-1)(q^{2}-q)^{2g-2} \right).
\end{align*}
To obtain the irreducible locus,
\begin{align*}
e(I) & = e(X_{0}^{g})-e(R) \\
& = (q^{3}-q) \left( (q^{3}-q)^{2g-2}+(q^{2}-1)^{2g-2}-(q^{2}-q)^{2g-2}+\frac{1}{2}q^{2g-1}((q-1)^{2g-2}+(q+1)^{2g-2}) \right.\\
& \left. -\frac{1}{2}((q+1)^{2g-1}-(q-1)^{2g-1})-q^{2g-2}-(q-1)(q^{2}-q)^{2g-2} \right)
\end{align*}
Finally,
\begin{align*}
e(\mathcal{M}_{\Id}^{g}) & = e(I)/e(\pgl) + e(\mathcal{M}_{\Id}^{red}) \\
& = (q^{3}-q)^{2g-2}+(q^{2}-1)^{2g-2}-q(q^{2}-q)^{2g-2} -q^{2g-2}  \\
& +\frac{1}{2}q^{2g-1}((q-1)^{2g-2}+(q+1)^{2g-2})+\frac{1}{2}q((q+1)^{2g-1}+(q-1)^{2g-1}) \\
\end{align*}
which completes the proof of Theorem \ref{thm:epolysmodulis}.

\subsection*{Topological consequences}
We state some consequences that can be derived from Theorem \ref{thm:epolysmodulis}. The first one is the Euler characteristic of $\mathcal{M}_{C}$, that can be obtained evaluating $e(\mathcal{M}_{C})$ at $q=1$. The value $\chi(\mathcal{M}_{-\Id})$ agrees with the result obtained in \cite{hausel-rvillegas:2008}.
\begin{cor}
Let $X$ be a complex curve of genus $g\geq 2$. The Euler characteristics of the character varieties $\mathcal{M}_{C}(\pgl)$ are
\begin{align*}
\chi (\mathcal{M}_{\Id}) & = 3\cdot 2^{2g-3}-1 \\
\chi (\mathcal{M}_{-\Id}) & = -2^{2g-3}  \\
\chi (\mathcal{M}_{J_{+}}) & = -2^{2g-2} \\
\chi (\mathcal{M}_{J_{-}}) & = 2^{2g-2} \\
\chi (\mathcal{M}_{\xi_{\lambda}}) & = 0
\end{align*}
\end{cor}

\begin{cor}
Let $X$ be a complex curve of genus $g\geq 2$. The spaces $\mathcal{M}_{\Id}$ and $\mathcal{M}_{-\Id}$ are of dimension $6g-6$ and the spaces $\mathcal{M}_{J_{+}}, \mathcal{M}_{J_{-}}, \mathcal{M}_{\xi_{\lambda}}$ are of dimension $6g-4$. All of them have a unique component of maximal dimension.
\end{cor}
\begin{proof}
The degree of $e(\mathcal{M}_{C})$ gives the dimension of the character variety and its leading coefficient the number of irreducible components of maximal dimension. The result follows from Theorem \ref{thm:epolysmodulis}.
\end{proof}

\section*{Mirror symmetry between $\sldos,\pgl$-character varieties}
The relation between mirror symmetry and character varieties for Langlands dual groups has been object of study in recent years. Clasically, if we look at $\mathcal{M}_{Dol}(PGL(n,\CC))$ and $\mathcal{M}_{Dol}(SL(n,\CC))$, both spaces can be regarded as fibrations over a common base, the Hitchin base $\mathcal{A}_{SL(n,\CC)}\cong \mathcal{A}_{PGL(n,\CC)}$, so that the generic fibres are dual abelian varieties. If the complex structure is changed, then the de Rham moduli spaces $\mathcal{M}_{dR}(SL(n,\CC))$, $\mathcal{M}_{dR}(PGL(n,\CC))$ become special Lagrangian fibrations that satisfy the conditions for the SYZ setup \cite{hausel-thaddeus:2003}. Inspired by these situations, the following conjecture for the Betti moduli space was proposed in \cite{hausel:2005}, called the Topological Mirror Test, which was proved in several cases ($n=4$, or prime):
\begin{conj} \label{conj:mirrorepoly}
 For all $d,e\in \ZZ$, such that $(d,n)=(e,n)=1$, we have
$$
e_{st}^{B^{e}}(x,y,\mathcal{M}_{B}^{d}(SL(n,\CC)))=e_{st}^{\hat{B}^{d}}(x,y,\mathcal{M}_{B}^{e}(PGL(n,\CC)))
$$
\end{conj}
where the stringy E-polynomial $e_{st}$ for orbifolds $X/\Gamma$ was defined  in 
\cite{hausel-thaddeus:2003}, 
$$
e_{st}(X/\Gamma,u,v)=\sum_{\left[ \gamma \right]}E(X^{\gamma}/C(\gamma), L_{\gamma}^{B},u,v)(uv)^{F(\gamma)},
$$
where $B$ is a flat unitary gerbe, $C(\gamma)$ denotes the centralizer of $\gamma$ in $G$, $L_{\gamma}$ is a line bundle associated to $B$ and $F(\gamma)$ is an integer called the fermionic shift. Stringy Hodge polynomials arose in string theory to take into account the singularities of a variety with orbifold singularites (more generally, Gorenstein singularities). They agree with the E-polynomial of a crepant resolution, when there is one \cite{kontsevich:1995}.
Conjecture \ref{conj:mirrorepoly} for $n=2$ and  $d$ odd was already proved in \cite{hausel-rvillegas:2008} and it agrees with our result:  $\mathcal{M}_{B}^{d}(\sldos)$ is $\mathcal{M}_{-\Id}(\sldos)$ in our notation and it is smooth, so the stringy E-polynomial concides with the ordinary one. The difference between the E-polynomials of the $\sldos$ and the $\pgl$-character varieties $e(\mathcal{M}_{-\Id})$ is precisely the stringy contribution coming from the non-trivial elements of $\Gamma=\ZZ_{2}^{2g}$:
\begin{align*}
e(\mathcal{M}_{-\Id}(\sldos))-e(\mathcal{M}_{-\Id}(\pgl)) & = \sum_{\left[ \gamma \right]\neq 1}e(X^{\gamma}/C(\gamma),L_{\gamma}^{B})q^{F(\gamma)} \\ & = (2^{2g}-1)q^{2g-2}\left( \frac{(q-1)^{2g-2}-(q+1)^{2g-2}}{2} \right)
\end{align*}
so that $e_{st}(\mathcal{M}_{-\Id}(\sldos))=e_{st}(\mathcal{M}_{-\Id}(\pgl))$.

In addition, an interesting picture occurs when we consider these stringy differences for the whole range of moduli spaces $\mathcal{M}_{C}(\pgl)$ and $\mathcal{M}_{C}(\sldos)$ for arbitrary holonomy $C$ at the puncture. The E-polynomials of the $\sldos$-character varieties were computed in \cite{mamu2}. We recall a particular result: the Hodge monodromy representation of $\overline{X}_{4}/\ZZ_2$, given by
$$
R(\oX_{4}/\ZZ_2)= a^{g}T+b^{g}S_{2}+c^{g}S_{-2}+d^{g}S_{0} \in R(\ZZ_2 \times \ZZ_2)[q]
$$
where $T$ is the trivial representation and $S_{\pm 2}$ are trivial over $\pm 2$ and non trivial over $\mp 2$ and $S_{0}=S_{2}\otimes S_{-2}$, with coefficients \cite[Theorem 7]{mamu2}
\begin{align*}
a^{g} & = (q^{3}-q)^{2g-1}+\frac{1}{2}q^{2g-1} \left( (q+1)^{2g-1}-(q-1)^{2g-1} \right), \\
b^{g} & = 2^{2g-1}(q^{2}-q)^{2g-1}-2^{2g-1}(q^{2}+q)^{2g-1} +\frac{1}{2}q^{2g-1} \left( (q+1)^{2g-1}-(q-1)^{2g-1} \right) ,\\
c^{g} & = 2^{2g-1}(q^{2}-q)^{2g-1}+2^{2g-1}(q^{2}+q)^{2g-1}-\frac{1}{2}q^{2g-1} \left( (q+1)^{2g-1}+(q-1)^{2g-1} \right), \\
d^{g} & = (q^{2}-1)^{2g-1}-\frac{1}{2}q^{2g-1} \left( (q+1)^{2g-1} +(q-1)^{2g-1} \right),
\end{align*}
Note that $a^{g}=\tilde{a}^{g}$ and $d^{g}=\tilde{b}^g$. If we look at the differences of the E-polynomials of the $\sldos$-representation spaces $\overline{X}_{i}$, $e_{i}^{g}$, and the their respective $\pgl$-versions $\overline{Z}_{i}$, $\tilde{e}_{i}^{g}$, we observe the following relation with the Hodge monodromy representation of $\overline{X}_{4}/\ZZ_2$
\begin{align}
e_{0}-\tilde{e}_{0} & = (q^3-q)(2^{2g}-1)\left(\frac{(q^{2}-q)^{2g-2}+(q^{2}+q)^{2g-2}}{2}\right)=qc^{g}+b^{g} \nonumber\\
e_{1}-\tilde{e}_{1} & = (q^3-q)(2^{2g}-1)\left(\frac{(q^{2}-q)^{2g-2}-(q^{2}+q)^{2g-2}}{2}\right)=qb^{g}+c^{g} \\
e_{2}-\tilde{e}_{2} & = (2^{2g}-1)\left(\frac{(q^{2}-q)^{2g-1}-(q^{2}+q)^{2g-1}}{2}\right)=b^{g} \nonumber\\
e_{3}-\tilde{e}_{3} & = (2^{2g}-1)\left(\frac{(q^{2}-q)^{2g-1}+(q^{2}+q)^{2g-1}}{2}\right)=c^{g} \nonumber \\
e_{4,\lambda}-\tilde{e}_{4,\lambda} & = (2^{2g}-1)(q^{2}-q)^{2g-1} =b^{g}+c^{g} \nonumber
\end{align}
We note that the differences have a simple description in terms of $b^{g}$ and $c^{g}$, the non-invariant terms in $R(\oX_{4}/\ZZ_2)$. There is also a certain symmetry in the expressions: the differences for $e_{0}$ and $e_{1}$ and for the non-semisimple cases $e_{2}$ and $e_{3}$ get interchanged under the permutation $b^{g}\leftrightarrow c^{g}$, that leaves the parabolic case unaffected. The permutation corresponds to the map $\tau(z)=-z$, that induces a map of Hodge monodromy representations $\tau^{\ast}(R(\oX_{4}/\ZZ_2)=a^{g}T+c^{g}S_{2}+b^{g}S_{-2}+d^{g}S_{0}$. We can summarize it in the following theorem.
\begin{thm} \label{thm:mirrorsymetrricdiffs}
Let $G=\sldos$, and let $\hat{G}=\pgl$ be its Langlands dual group. The differences between the E-polynomials of the $G,\hat{G}$-character varieties associated to punctured Riemann surfaces can be expressed in terms of the coefficients of the Hodge monodromy representation of the fibration $\oX_{4}/\ZZ_2 \rightarrow \CC-\lbrace \pm 2 \rbrace$,
$$
R(\oX_{4}/\ZZ_2)=a^{g}T+b^{g}S_{2}+c^{g}S_{-2}+d^{g}S_{0}
$$
These differences are the following: 
\begin{align}
e_{0}-\tilde{e}_{0} & = (q^3-q)(2^{2g}-1)\left(\frac{(q^{2}-q)^{2g-2}+(q^{2}+q)^{2g-2}}{2}\right)=qc^{g}+b^{g} \nonumber\\
e_{1}-\tilde{e}_{1} & = (q^3-q)(2^{2g}-1)\left(\frac{(q^{2}-q)^{2g-2}-(q^{2}+q)^{2g-2}}{2}\right)=qb^{g}+c^{g} \\
e_{2}-\tilde{e}_{2} & = (2^{2g}-1)\left(\frac{(q^{2}-q)^{2g-1}-(q^{2}+q)^{2g-1}}{2}\right)=b^{g} \nonumber\\
e_{3}-\tilde{e}_{3} & = (2^{2g}-1)\left(\frac{(q^{2}-q)^{2g-1}+(q^{2}+q)^{2g-1}}{2}\right)=c^{g} \nonumber \\
e_{4,\lambda}-\tilde{e}_{4,\lambda} & = (2^{2g}-1)(q^{2}-q)^{2g-1} =b^{g}+c^{g} \nonumber
\end{align}
where $e_{i},\tilde{e}_{i}, i=1\ldots 4$ are the E-polynomials of the $\sldos$, $\pgl$-representation spaces for $C=\Id, -\Id, J_{+},J_{-}, \xi_{\lambda}$ respectively.
\end{thm}
\begin{rem}
If we write $E_{i}=e_{i}-\tilde{e}_{i}, i=1\ldots 4$ for the differences between the $\pgl$ and the $\sldos$ cases, it is immediate to deduce that
$$
E_{0}+(q-1)E_{2}=qE_{4}, \qquad E_{1}+(q-1)E_{3}=qE_{4,\lambda}.
$$

The second identity is also immediate from Corollary \ref{cor:Hauselrel} and its $\sldos$-version in \cite{mamu2}, whereas the first can be derived from the identities
$$
e_{0}+(q-1)e_{2}=qe_{4,\lambda} + (q^{2}-q)^{2g}, \qquad \tilde{e}_{0}+(q-1)\tilde{e}_{2}=q\tilde{e}_{4,\lambda}+ (q^{2}-q)^{2g}.
$$

\end{rem}
For $G=SL(n,\CC)$ and $G=PGL(n,\CC)$, it is natural to conjecture an extension of Theorem \ref{thm:mirrorsymetrricdiffs} for arbitrary $n$. In that case, we would have that the representation spaces $X_{C}$ that correspond to the moduli spaces $\mathcal{M}_{C}(SL(n,\CC))$ fibre over $\mathcal{M}(F_{1},SL(n,\CC))\cong \CC^{n-1}$,
\begin{align*}
\oX^{g}(SL(n,\CC))  & \rightarrow \CC^{n-1} \\
(A_{1},B_{1},\ldots, A_{g},B_{g}) & \mapsto \text{charpoly}(C)= \text{charpoly} \left( \prod_{i=1}^{2g}[A_{i},B_{i}] \right)
\end{align*}
the isomorphism being given by the coefficients of the characteristic polynomial. The character varieties corresponding to the Higgs bundles moduli space, where $C=\zeta_{} \Id$ and $\zeta$ is a primitive $n$-th root of unity,  lie inside the discriminant locus $\triangle \subset \CC^{n-1}$, where the multiplicity of some of the eigenvalues is greater than one and there are non-semisimple conjugacy classes over those points. We can consider the set of parabolic character varieties with semisimple holonomy $C$ of multiplicities of type $(1,\ldots, 1)$, that is, such that all its eigenvalues are different. If we write $\xi_{\boldsymbol{\lambda}}=\text{Diag}(\boldsymbol{\lambda})$ for the diagonal matrix $C$ of eigenvalues $\boldsymbol{\lambda}=(\lambda_{1},\ldots, \lambda_{n-1})$  (recall that $\lambda_{n}=\lambda_{1}^{-1}\ldots \lambda_{n-1}^{-1}$, since $C\in SL(n,\CC)$), we get a fibration
$$
\oX_{\xi_{\boldsymbol{\lambda}}} \longrightarrow B
$$
where $B= (\CC^{\ast})^{n-1} - \bigcup_{i<j} \lbrace \lambda_{i}=\lambda_{j} \rbrace$. The symmetric group $S_{n}$ acts on the base permuting the eigenvectors and on the total space of the fibration $\oX_{\xi_{\boldsymbol{\lambda}}}$ by conjugation by a suitable permutation matrix. This produces a fibration $\oX_{\xi_{\boldsymbol{\lambda}}}/S_{n} \rightarrow B/S_{n}$, which has Hodge monodromy representation $R(\oX_{\xi_{\boldsymbol{\lambda}}})$. We conjecture
\begin{conj} \label{conj:Hodgemonmirrorsymmetry}
The difference $e(\oX_{C}(SL(n,\CC)))-e(\oX_{C}(PGL(n,\CC)))$ is described by the coefficients of the Hodge monodromy representation $R(\oX_{\xi_{\boldsymbol{\lambda}}}/S_{n})$.
\end{conj}
It is expected that such differences obey a $S_{n}$-symmetric pattern as it occurred in the $n=2$ case.
    
  

\end{document}